\documentclass{amsart}

\usepackage{amssymb,amsmath,ifthen}
\usepackage{stmaryrd}
\usepackage{yhmath}
\usepackage{latexsym}
\usepackage{amsthm}
\usepackage{amscd}
\usepackage{mathrsfs}
\usepackage[all]{xy}
\usepackage{mathtools}
\usepackage{arydshln}
\usepackage{url}
\usepackage{bbm}

\newcommand{\ur}{\mathrm{ur}}
\newcommand{\rs}{\mathrm{rs}}
\newcommand{\trs}{\theta\mathchar`-\mathrm{rs}}
\newcommand{\srs}{\mathrm{srs}}
\newcommand{\strs}{\theta\mathchar`-\mathrm{srs}}

\newcommand{\lan}{\langle}
\newcommand{\ran}{\rangle}
\newcommand{\ol}{\overline}

\newcommand{\mcO}{\mathcal{O}}
\newcommand{\mfp}{\mathfrak{p}}

\newcommand{\Z}{\mathbb{Z}}

\newcommand{\R}{\mathbb{R}}

\newcommand{\C}{\mathbb{C}}
\newcommand{\Gm}{\mathbb{G}_{\mathrm{m}}}
\newcommand{\T}{\mathbf{T}}

\newcommand{\bfS}{\mathbf{S}}
\newcommand{\bfG}{\mathbf{G}}
\newcommand{\bfZ}{\mathbf{Z}}

\newcommand{\w}{\widetilde}

\newcommand{\U}{\mathrm{U}}
\newcommand{\G}{\mathrm{G}}

\newcommand{\ra}{\rightarrow}

\DeclareMathOperator{\tr}{tr}

\DeclareMathOperator{\GL}{GL}

\DeclareMathOperator{\SO}{SO}

\DeclareMathOperator{\SL}{SL}

\DeclareMathOperator{\cInd}{c-Ind}

\DeclareMathOperator{\Hom}{Hom}
\DeclareMathOperator{\Gal}{Gal}

\DeclareMathOperator{\Kl}{Kl}
\DeclareMathOperator{\Tr}{Tr}
\DeclareMathOperator{\Nr}{Nr}
\DeclareMathOperator{\Cent}{Cent}
\DeclareMathOperator{\diag}{diag}

\DeclareMathOperator{\SSC}{SSC}

\theoremstyle{plain}
\newtheorem{thm}{Theorem}[section]
\newtheorem*{thm*}{Theorem}
\newtheorem{prop}[thm]{Proposition}
\newtheorem{lem}[thm]{Lemma}
\newtheorem{cor}[thm]{Corollary}

\theoremstyle{definition}
\newtheorem{defn}[thm]{Definition}

\theoremstyle{remark}
\newtheorem{rem}[thm]{Remark}
\newtheorem*{claim*}{Claim}

\SelectTips{cm}{11}
\title{Endoscopic lifting of simple supercuspidal representations of unramified $\U_{N}$ to $\GL_{N}$}
\author{Masao Oi}
\address{Graduate School of Mathematical Sciences, 
the University of Tokyo, 3-8-1 Komaba, Meguro-ku, Tokyo 153-8914, Japan.}
\email{masaooi@ms.u-tokyo.ac.jp}

\begin{document}

\begin{abstract}
We compute the characters of simple supercuspidal representations of twisted $\GL_{N}$ and unramified quasi-split unitary group $\U_{N}$ over a $p$-adic field.
Comparing them by the endoscopic character relation, we determine the endoscopic lifts of simple supercuspidal representations of $\U_{N}$ to $\GL_{N}$, under the assumption that $p$ is not equal to $2$.
\end{abstract}


\maketitle

\section{Introduction}
The purpose of this paper is to describe the local Langlands correspondence explicitly for simple supercuspidal representations of unramified unitary groups.

Let us first recall the local Langlands correspondence for quasi-split classical groups, which have been established by Arthur (\cite{MR3135650}) and Mok (\cite{MR3338302}).
Let $\bfG$ be a quasi-split classical group (that is, a symplectic, special orthogonal, or unitary group) over a $p$-adic field $F$.
Let $\Pi(\bfG)$ be the set of equivalence classes of irreducible smooth representations of $G:=\bfG(F)$, 
and $\Phi(\bfG)$ the set of equivalence classes of $L$-parameters of $\bfG$.
Here an $L$-parameter of $\bfG$ is an admissible homomorphism from the Weil-Deligne group $W_{F}\times\SL_{2}(\C)$ to the $L$-group ${}^L\bfG=\widehat{G}\rtimes W_{F}$ of $\bfG$.
Then the \textit{local Langlands correspondence for $\bfG$} gives a natural map from the set $\Pi(\bfG)$ to the set $\Phi(\bfG)$ with finite fibers (called $L$-packets).

Here recall that the naturality of this map is formulated by the theory of \textit{endoscopy}.
More precisely, we first note that $\bfG$ can be regarded as an endoscopic group of $\widetilde{\bfG}=\GL_{N}$ (or a Weil restriction of $\GL_{N}$) for some $N$.
In particular, we have an embedding from ${}^{L}\bfG$ to ${}^{L}\!\GL_{N}$.
Therefore, for an $L$-parameter $\phi$ of $\bfG$, we can regard it as an $L$-parameter of $\widetilde{\bfG}$.
Thus $\phi$ defines $L$-packets of $\bfG$ and $\widetilde{\bfG}$:
\[
\xymatrix{
\Pi(\widetilde{\bfG}) \supset \Pi_{\phi}^{\widetilde{\bfG}}:=(\text{$L$-packet of $\widetilde{\bfG}$ for $\phi$})& \ar@{<~>}[r]^-{\text{LLC for $\widetilde{\bfG}$}} && W_F\times\SL_2(\C) \ar[rd]_-{\phi} \ar[r] & {}^{L}\widetilde{\bfG}\\
\Pi(\bfG) \supset \Pi_{\phi}^{\bfG}:=(\text{$L$-packet of $\bfG$ for $\phi$})& \ar@{<~>}[r]^-{\text{LLC for $\bfG$}} &&& {}^{L}\bfG \ar@{^{(}->}[u] \\
}
\]
In this situation, the $L$-packet $\Pi_{\phi}^{\widetilde{\bfG}}$ of $\widetilde{\bfG}$ is called the \textit{endoscopic lift} of the $L$-packet $\Pi_{\phi}^{\bfG}$ of $\bfG$.
Here we remark that the local Langlands correspondence for $\widetilde{\bfG}$ have been established by Harris and Taylor (\cite{MR1876802}), and in this case the map is bijective (i.e., every $L$-packet is a singleton).
Then, if $\phi$ is a tempered $L$-parameter, $\Pi_{\phi}^{\bfG}$ and $\Pi_{\phi}^{\widetilde{\bfG}}$ satisfy the \textit{endoscopic character relation}, which is an equality between the characters of representations belonging to $\Pi_{\phi}^{\bfG}$ and the twisted character of $\Pi_{\phi}^{\widetilde{\bfG}}$.
This is the characterization of the local Langlands correspondence for $\bfG$.

Our interest is to describe the local Langlands correspondence for $\bfG$ explicitly.
Then, from the above characterization of the correspondence, it is reduced to describing that for $\widetilde{\bfG}$ and determining the endoscopic liftings from $\bfG$ to $\widetilde{\bfG}$.
For this, it is important to compute the characters of representations.

In our past paper \cite{Oi:2016}, in the case of $(\widetilde{\bfG},\bfG)=(\GL_{2n}, \SO_{2n+1})$, we considered this problem for \textit{simple supercuspidal} representations, which are introduced in \cite{MR2730575} and \cite{MR3164986}.
These are supercuspidal representations having the minimal positive depth, and have been studied by many people in the context of finding an explicit description of the local Langlands correspondence.
In particular, their $L$-parameters have been described explicitly in the cases of general linear groups, by the works of Bushnell--Henniart (\cite{MR2148193}) and Imai--Tsushima (\cite{Imai:2015aa}).
In \cite{Oi:2016}, under the assumption that $p$ is not equal to $2$, we proved that a simple supercuspidal representation of $\SO_{2n+1}(F)$ constitutes an $L$-packet and that its lift to $\GL_{2n}(F)$ is again simple supercuspidal, and determined it explicitly.

In this paper, under the same assumption that $p$ is not equal to $2$, we consider this lifting problem for simple supercuspidal representations in the case of $(\widetilde{\bfG}, \bfG)=(\G_{N}, \U_{N})$.
Here, $\G_{N}$ is the Weil restriction $\mathrm{Res}_{E/F}\GL_{N}$ of $\GL_{N}$ from the unramified quadratic extension $E$ of $F$ to $F$, and $\U_{N}$ is the quasi-split unitary group with respect to $E/F$ in $N$ variables.
Then there exist two kinds of $L$-embeddings $\{\xi_{\kappa}\}_{\kappa\in\{\pm1\}}$ from ${}^{L}\U_{N}$ to ${}^{L}\G_{N}$ (see Section \ref{sec:endo}.1 for the details).
Hence we can consider the endoscopic liftings from $\U_{N}$ to $\G_{N}$ via these two kinds of $L$-embeddings.
On the other hand, after making some choices (for example, fixing a uniformizer of $F$), we can parametrize the set of equivalence classes of simple supercuspidal representations of these groups explicitly. 
We denote the parametrizing sets for them by $\SSC(\G_{N})$ and $\SSC(\U_{N})$ (see Sections \ref{sec:ssc}.2 and \ref{sec:ssc}.3 for the details).
Then our main theorem is stated as follows:

\begin{thm}[Theorem \ref{thm:main}]\label{thm:intro}
We assume that $p$ is not equal to $2$.
Let $X\in\SSC(\U_{N})$, and we consider the corresponding simple supercuspidal representation $\pi_{X}^{\U_{N}}$ of $\U_{N}(F)$.

\begin{enumerate}
\item
The $L$-packet of $\U_{N}$ containing the simple supercuspidal representation $\pi_{X}^{\U_{N}}$ is a singleton.
In particular, the character of $\pi_{X}^{\U_{N}}$ is stable.
\item
The endoscopic lift of the simple supercuspidal representation $\pi_{X}^{\U_{N}}$ of $\U_{N}(F)$ to $\G_{N}(F)$ via the $L$-embedding $\xi_{\kappa}$ is again simple supercuspidal.
Furthermore, we can describe the parameter $Y\in\SSC(\G_{N})$ corresponding to the lifted simple supercuspidal representation of $\G_{N}(F)$ explicitly in terms of $X$ and $\kappa$.
\end{enumerate}
\end{thm}

We note that, by combining this result with a description of $L$-parameters of simple supercuspidal representations of general linear groups due to Bushnell--Henniart and Imai--Tsushima, we get a complete description of the local Langlands correspondence for simple supercuspidal representations of $\U_{N}(F)$.

We explain the outline of our proof.
Our strategy is basically the same as in \cite{Oi:2016}.
That is, we descend simple supercuspidal representations of $\G_{N}(F)$ instead of lifting simple supercuspidal representations of $\U_{N}(F)$.
In order to determine the descended representations, we compute the endoscopic character relation explicitly.
In principle, representations can be recovered from their character completely.
However, it is not practical to compute their character at the all elements.
Therefore the key idea of our proof is to compute the characters only at ``essential'' elements.
Here we emphasize that the actual computations in the details of our arguments (such as the characters of representations or the  \textit{norm correspondence}, which is a correspondence between the stable conjugacy classes of $\w{\bfG}$ and $\bfG$ appearing in the endoscopic character relation) heavily depend on the individual groups $\w{\bfG}$ and $\bfG$.
Therefore we can not carry out such computations in a uniform way, and have to investigate the structures of our groups $\G_{N}$ and $\U_{N}$ carefully.

We now explain each step of our proof more precisely.
We first take a conjugate self-dual simple supercuspidal representation $\pi_{Y}^{\G_{N}}$ of $\G_{N}(F)$.
Then, from a property of the local Langlands correspondence, its $L$-parameter is conjugate self-dual.
In particular, the $L$-parameter factors through $\xi_{\kappa}$ for either $\kappa=+1$ or $\kappa=-1$.
To determine $\kappa$, we investigate the behavior of the twisted character of $\pi_{Y}^{\G_{N}}$ at some special elements of $\G_{N}(F)$.
By using the twisted character formula for supercuspidal representations, we write these character values explicitly in terms of the Kloosterman sums.
Combining the endoscopic character relation with properties of the Kloosterman sums, we can determine $\kappa$ (Section \ref{sec:main}.2).
We note that the arguments in this step are completely different from those of \cite{Oi:2016}.
In \cite{Oi:2016}, in order to show that the $L$-parameter of a simple supercuspidal representation of $\GL_{2n}(F)$ factors through the $L$-group of $\SO_{2n+1}$, we need Mieda's result on the parity of self-dual supercuspidal representations (\cite{Mieda:2016}), which is based on geometric techniques.

Then we get the $L$-packet of $\U_{N}$, and know that this is a singleton consisting of a supercuspidal representation $\pi_{\U_{N}}$ by a result of M{\oe}glin (\cite{MR2366373}).
The next step is to show that this unique representation $\pi_{\U_{N}}$ in the $L$-packet is simple supercuspidal.
To prove this, we again use the endoscopic character relation. 
By the endoscopic character relation, we can express the character of $\pi_{\U_{N}}$ in terms of the twisted character of $\pi_{Y}^{\G_{N}}$, which is already computed.
From this, we can show that $\pi_{\U_{N}}$ is either simple supercuspidal or depth-zero supercuspidal.
To eliminate the possibility that $\pi_{\U_{N}}$ is depth-zero supercuspidal, we next compute the character of depth-zero supercuspidal representations, and compare them.

The final step is to determine the parameter $X\in\SSC(\U_{N})$ corresponding to the representation $\pi_{\U_{N}}$.
However, by using several properties of Kloosterman sums, we can easily do this.
This completes the proof.

We finally comment on related works.
We believe that our method is basically valid for other representations of other groups as long as the endoscopic character relation exists and we can compute it at ``sufficiently'' many elements.
For example, as mentioned above, we can determine the endoscopic lift of simple supercuspidal representations of $\SO_{2n+1}(F)$ to $\GL_{2n}(F)$ by the same idea.
However, for the cases of symplectic group or unramified even orthogonal groups, the $L$-packets of simple supercuspidal representations of these groups are not singleton.
Therefore, in addition to the arguments used in this paper, we have to investigate the standard endoscopy for those $L$-packets.
On the other hand, in the case of ramified even orthogonal groups, the Kottwitz--Shelstad transfer factors are not trivial.
Thus the endoscopic character relations are more complicated than this paper's one.
Hence, also in this case, in addition to the techniques in this paper, we need further properties of the local Langlands correspondence.
For these groups, the simple supercuspidal $L$-packets are investigated in detail in \cite{Oi:2018}.
Other than these cases, for example, Adler and Lansky determined the endoscopic lifts (base change) for depth-zero supercuspidal representations of the quasi-split $\U_{3}$ (\cite{MR2167974,MR2657691}), and Blasco determined the endoscopic lifts of supercuspidal representations of the quasi-split $\U_{2}$ (\cite{MR2423457}) and ``very cuspidal'' representations of the quasi-split $\U_{3}$ (\cite{MR2680819}), as a consequence of an explicit computation of the character relation.


We explain the organization of this paper.
In Section \ref{sec:ssc}, we review some fundamental properties about Iwahori subgroups and simple supercuspidal representations, and explain parametrizations of simple supercuspidal representations of $\G_{N}(F)$ and $\U_{N}(F)$.
In Section \ref{sec:char}, we compute the (twisted) characters of simple supercuspidal representations of $\G_{N}(F)$ and $\U_{N}(F)$ at some special elements.
In Section \ref{sec:endo}, we investigate the norm correspondence and the transfer factor for $\G_{N}$ and $\U_{N}$.
In Section \ref{sec:main}, we first recall the endoscopic character relation in \cite{MR3338302}.
Then we determine the endoscopic lifts of simple supercuspidal representations by combining the endoscopic character relation with the results in Sections \ref{sec:char} and \ref{sec:endo}.

\medbreak
\noindent{\bfseries Acknowledgment.}\quad
This paper is a part of my master's thesis.
I would like to thank to my advisor Yoichi Mieda for his support and encouragement.
He carefully read the draft version of this paper, and gave me many valuable suggestions.
I am grateful to Naoki Imai for his helpful comments.

This work was carried out with the support from the Program for Leading Graduate Schools, MEXT, Japan.
This work was also supported by JSPS Research Fellowship for Young Scientists and KAKENHI Grant Number 17J05451.

\setcounter{tocdepth}{2}
\tableofcontents

\medbreak
\noindent{\bfseries Notation.}\quad
\begin{description}
\item[$p$-adic field]
Let $p$ be an odd prime number.
We fix a $p$-adic field $F$.
We denote its ring of integers, its maximal ideal, and its residue field by $\mcO$, $\mfp$, and $k$, respectively.
We fix a uniformizer $\varpi$ of $F$.
For $x \in \mcO$, $\bar{x}$ denotes the image of $x$ in $k$.
We often regard an element of $k^{\times}$ as an element of $F^{\times}$ by the Teichm{\"u}ller lift.
We denote the Weil group of $F$ by $W_{F}$.

\item[unramified quadratic extension of $F$]
Let $E$ be an unramified quadratic extension of $F$.
We write $c$ for the nontrivial element of $\Gal(E/F)$.
We denote its ring of integers, its maximal ideal, and its residue field by $\mcO_{E}$, $\mfp_{E}$, and $\tilde{k}$, respectively.

\item[quadratic extension of $k$]
We denote the norm and the trace with respect to the quadratic extension $\tilde{k}/k$ shortly by $\Nr$ and $\Tr$.
We set $\tilde{k}^{1}$ and $\tilde{k}^{0}$ to be the kernel of $\Nr$ and $\Tr$, respectively:
\[
\tilde{k}^{1}:= \mathrm{Ker}(\Nr\colon \tilde{k}^{\times}\rightarrow k^{\times}),\quad
\tilde{k}^{0}:=\mathrm{Ker}(\Tr\colon \tilde{k}\rightarrow k).
\]
We fix a square root $\epsilon\in\tilde{k}$ of a non-square element of $k^{\times}$.
Then note that we have $\tilde{k}^{0}=\epsilon k$.
Finally, we often identify $\tilde{k}^{1}$ with $\tilde{k}^{\times}/k^{\times}$ via the isomorphism 
\[
\tilde{k}^{\times}/k^{\times} \ra \tilde{k}^{1};\quad z\mapsto z/c(z).
\]

\item[additive character]
Throughout this paper, we fix an additive character $\psi$ on $F$ of level one.
Then its restriction $\psi|_{\mcO}$ to $\mcO$ induces a nontrivial additive character on $k$.
We denote it by $\psi$ again.
Furthermore, we denote the additive character $\psi\circ\Tr$ on $\tilde{k}$ by $\tilde{\psi}$.

\item[algebraic group]
For an algebraic group $\bfG$ over $F$, we denote its $F$-rational points $\bfG(F)$ by $G$.
When $\bfG$ is abelian, we write $G(q)$ for the subgroup of $G$ consisting of the elements of finite orders.
For a connected reductive group $\bfG$ over $F$, we denote its Langlands dual group and the $L$-group by $\widehat{G}$ and ${}^{L}\bfG$, respectively.
For an algebraic group $\mathbf{T}$ over $F$ or $\C$, we write $X^\ast(\mathbf{T})$ for its absolute character group and $X_{\ast}(\mathbf{T})$ for its absolute cocharacter group.

\item[matrix]
We denote the identity matrix of size $N$ by $I_{N}$.
We set $J_{N}$ to be the anti-diagonal matrix whose $(i, N+1-i)$-th entry is given by $(-1)^{i-1}$:
\[
J_{N}:=
\begin{pmatrix} 
&&&1\\
&&-1&\\
&\adots&&\\
(-1)^{N-1}&&&
\end{pmatrix}.
\]
\end{description}

\section{Simple supercuspidal representations}\label{sec:ssc}

\subsection{Iwahori subgroups and simple supercuspidal representations}
We first recall the definition of simple supercuspidal representations.
See Section 2 in \cite{Oi:2018} for the details of the arguments in this section.

Let $\bfG$ be a quasi-split connected reductive group over $F$, and we assume that $\bfG$ is tamely ramified over $F$.
We write $\bfZ_{\bfG}$ for the center of $\bfG$.
We fix a maximal $F$-split torus $\bfS_{\bfG}$ of $\bfG$.
Then we get a relative root (resp.\ affine root) system $\Phi_{\bfG}$ (resp.\ $\Psi_{\bfG}$).
We fix a Borel subgroup of $\bfG$ which contains $\bfS_{\bfG}$ and is defined over $F$, and write $\Delta_{\bfG}$ for the corresponding root basis of $\Phi_{\bfG}$.
By choosing an alcove $\mathcal{C}$ of an apartment $\mathcal{A}(\bfG,\bfS_{\bfG})$ of the Bruhat--Tits building of $\bfG$, we get the corresponding affine basis $\Pi_{\bfG}$ of $\Psi_{\bfG}$ and the Iwahori subgroup $I_{\bfG}$ of $G$.
Furthermore, by the theory of the Moy--Prasad filtration, we have a filtration of $I_{\bfG}$ associated to the barycenter of the alcove $\mathcal{C}$.
We denote the first two steps of this filtration by $I_{\bfG}^{+}$ and $I_{\bfG}^{++}$.

Let $\T_{\bfG}$ be the centralizer of $\bfS_{\bfG}$ in $\bfG$, $T_{\bfG}^{0}$ the maximal compact subgroup of $T_{\bfG}=\T_{\bfG}(F)$, and $T_{\bfG}^{1}$ the pro-unipotent radical of $T_{\bfG}^{0}$.
Then recall that the groups $I_{\bfG}$, $I_{\bfG}^{+}$, and $I_{\bfG}^{++}$ are described explicitly in terms of affine root subgroups as follows:
\begin{align*}
I_{\bfG} &:= \lan T_{\bfG}^{0}, U_\alpha \mid \alpha \in \Psi_{\bfG}^+\ran,\\
I_{\bfG}^{+} &:= \lan T_{\bfG}^{1}, U_\alpha \mid \alpha \in \Psi_{\bfG}^+\ran \text{, and}\\
I_{\bfG}^{++} &:= \lan T_{\bfG}^{1}, U_\alpha \mid \alpha \in \Psi_{\bfG}^+ \setminus \Pi_{\bfG}\ran,
\end{align*}
where $\Psi_{\bfG}^{+}$ is the set of positive affine roots, and $U_{\alpha}$ is the affine root subgroup of $G$ associated to $\alpha$.
Then the quotient $I_{\bfG}/I_{\bfG}^{+}$ is identified with the group $T_{\bfG}(q)$ consisting of elements of $T_{\bfG}$ of finite orders.
Moreover, if we put
\[
V_{\bfG}:=I_{\bfG}^{+}/I^{++}_{\bfG},
\]
then it has a structure of $k$-vector space and decomposes into a direct sum of $\dot{\alpha}$-isotypic parts with respect to the $S_{\bfG}(q)$-action, where $\dot{\alpha}$ is the gradient of a simple affine root $\alpha\in\Pi_{\bfG}$:
\[
V_{\bfG}\cong\bigoplus_{\alpha\in\Pi_{\bfG}} V_{\bfG}(\dot{\alpha}).
\]

For an element $g$ in $I_{\bfG}^{+}$, we call the image of $g$ in $V_{\bfG}(\dot{\alpha})$ the \textit{simple affine component} with respect to $\alpha$.
If $g\in G$ belongs to $I_{\bfG}^{+}$ and every simple affine component is not zero, then we say that $g$ is \textit{affine generic}.

For a character $\chi$ of $Z_{\bfG}I_{\bfG}^{+}$, we say that $\chi$ is \textit{affine generic} if $\chi$ satisfies the following two conditions:
\begin{itemize}
\item $\chi|_{I_{\bfG}^{+}}$ is trivial on $I_{\bfG}^{++}$ (hence factors through $V_{\bfG}$), and
\item $\chi|_{I_{\bfG}^{+}}$ is not trivial on $V_{\bfG}(\dot{\alpha})$ for every $\alpha\in\Pi_{\bfG}$.
\end{itemize}
Let $\chi$ be an affine generic character of $Z_{\bfG}I_{\bfG}^{+}$.
Then we define the normalizer $N_{G}(I_{\bfG}^{+};\chi)$ of $\chi$ as follows:
\[
N_{G}(I_{\bfG}^{+};\chi):=\{n\in N_{G}(I_{\bfG}^{+}) \mid \chi^{n}=\chi\}.
\]
Here $N_{G}(I_{\bfG}^{+})$ is the normalizer of $I_{\bfG}^{+}$ in $G$, and $\chi^{n}$ is the twist of $\chi$ via $n$ defined by
\[
\chi^{n}(x):=\chi(nxn^{-1}).
\]
Then, for any irreducible constituent $\tilde{\chi}$ of $\cInd_{Z_{\bfG}I_{\bfG}^{+}}^{N_{G}(I_{\bfG}^{+};\chi)}\chi$, the representation 
\[
\cInd_{N_{G}(I_{\bfG}^{+};\chi)}^{G}\tilde{\chi}
\]
is irreducible, hence supercuspidal.
We call the irreducible supercuspidal representations of $G$ obtained in this way \textit{simple supercuspidal representations} of $G$.

In the rest of this section, for the groups considered in this paper, we explain the following data:
\begin{itemize}
\item
the (affine) root system and a choice of its (affine) root basis,
\item
the structure of the graded quotient of the Moy--Prasad filtration of the Iwahori subgroup,
\item
the structure of the normalizer group for each affine generic character, and
\item
a parametrization of the equivalence classes of simple supercuspidal representations.
\end{itemize}

\subsection{Parametrization: the case of $\G_{N}$}
In this subsection, we consider the case of 
\[
\G_{N}:=\mathrm{Res}_{E/F}(\GL_{N}).
\]
Note that we have $\G_{N}(F)=\GL_{N}(E)$.

We choose $\bfS_{\G_{N}}$ to be the subgroup of diagonal matrices whose $F$-valued points are given by
\[
S_{\G_{N}}=\{\diag(t_1, \ldots, t_{N}) \in \GL_{N}(E) \mid t_i\in F^{\times} \}.
\]
Then its centralizer $\T_{\G_{N}}$ is a maximal torus of $\G_{N}$ whose $F$-valued points are given by
\[
T_{\G_{N}}=\{ \diag(t_1, \ldots, t_{N}) \in \GL_{N}(E) \mid t_i\in E^{\times} \}.
\]
The relative root system and the affine root system are given by
\[
\Phi_{\G_{N}}=\{\pm(e_i-e_j) \mid 1\leq i<j\leq N\},
\]
\[
\Psi_{\G_{N}}=\{a+r \mid a\in\Phi_{\G_{N}}, r\in\Z\}.
\]
We take the root basis and the affine root basis to be
\[
\Delta_{\G_{N}}=\{e_1-e_2, \ldots, e_{N-1}-e_{N}\},
\]
\[
\Pi_{\G_{N}}=\{e_1-e_2, \ldots, e_{N-1}-e_{N}, e_{N}-e_1+1\}.
\]
For $x=(x_{ij})_{ij} \in I_{\G_{N}}^{+}$, we regard its simple affine components as an element of $\tilde{k}^{\oplus N}$ by
\begin{align*}
I_{\G_{N}}^{+}/I_{\G_{N}}^{++} &\cong \bigoplus_{\alpha\in\Pi_{\G_{N}}}V_{\G_{N}}(\dot{\alpha}) \cong \tilde{k}^{\oplus N} \\
(x_{ij})_{ij} &\mapsto \left(\ol{x_{1 2}}, \ldots, \ol{x_{N-1, N}}, \ol{x_{N 1}\varpi^{-1}}\right).
\end{align*}

For $a\in \tilde{k}^{\times}$, we put
\[
\varphi_a := 
\begin{pmatrix}
0 & I_{N-1} \\
\varpi a & 0 
\end{pmatrix} \in \G_{N}(F)=\GL_{N}(E).
\]
Then the normalizer of the Iwahori subgroup $I_{\G_{N}}$ is described as follows:
\begin{lem}\label{lem:N-Iwahori}
For any $a$, we have 
$N_{\G_{N}(F)}(I_{\G_{N}})=N_{\G_{N}(F)}(I_{\G_{N}}^{+})=Z_{\G_{N}}I_{\G_{N}}\lan\varphi_{a}\ran$.
\end{lem}

\begin{proof}
Let $\GL_{N,E}$ be the split general linear group of size $N$ over $E$.
Then, since the Iwahori subgroup $I_{\G_{N}}$ (resp.\ its pro-unipotent radical $I_{\G_{N}}^{+}$) of $\G_{N}(F)$ can be regarded as the standard Iwahori subgroup $I_{\GL_{N,E}}$ (resp.\ its pro-unipotent radical $I_{\GL_{N,E}}^{+}$) of $\GL_{N,E}(E)$, it is enough to show that
\[
N_{\GL_{N,E}(E)}(I_{\GL_{N,E}})=N_{\GL_{N,E}(E)}(I_{\GL_{N,E}}^{+})=Z_{\GL_{N,E}}I_{\GL_{N,E}}\lan\varphi_{a}\ran.
\]

By \cite[Lemma 2.4]{Oi:2018}, to show these equalities, it suffices to check that, for any $a$, $\lan\varphi_{a}\ran$ is a set of representatives of the subgroup $\widetilde{\Omega}$ of the Iwahori Weyl group $\widetilde{W}$ consisting of the elements normalizing $I_{\GL_{N,E}}$ (see \cite[(1.5)]{MR3481263} and also \cite[Proposition 2.3]{Oi:2018}).
We can easily check that $\lan\varphi_{a}\ran$ normalizes the Iwahori subgroup $I_{\GL_{N,E}}$.
Thus, since $\widetilde{\Omega}$ maps to $X^{\ast}(Z_{\widehat{\GL_{N,E}}})$ isomorphically under the Kottwitz homomorphism
\[
\kappa_{\GL_{N,E}}\colon
\GL_{N,E}(E) \rightarrow X^{\ast}(Z_{\widehat{\GL_{N,E}}})
\]
(see \cite[(1.5)]{MR3481263} and also \cite[Proposition 2.3]{Oi:2018}), it is enough to show that $\lan\varphi_{a}\ran$ surjects on $X^{\ast}(Z_{\widehat{\GL_{N,E}}})$ under the Kottwitz homomorphism $\kappa_{\GL_{N,E}}$.

In order to show this, let us recall the definition of the Kottwitz homomorphism $\kappa_{\GL_{N,E}}$.
We first note that the dual of the quotient $\GL_{N,E}/\SL_{N,E}\cong\Gm$ can be identified with the center $Z_{\widehat{\GL_{N,E}}}$ of $\widehat{\GL_{N,E}}$.
Then $\kappa_{\GL_{N,E}}$ is defined to be the composition of the determinant map
\[
\det\colon \GL_{N,E}(E) \rightarrow (\GL_{N,E}/\SL_{N,E})(E)\cong\Gm(E)
\]
and the map
\[
\Gm(E) \rightarrow \Hom(X^{\ast}(\Gm), \Z) \cong X_{\ast}(\Gm) \cong X^{\ast}(Z_{\widehat{\GL_{N,E}}})
\]
(see Section 7.2 and the diagram (7.4.1) in \cite{MR1485921} for the details).
Here the first part of the latter map sends an element $t\in\Gm(E)$ to the element of $\Hom(X^{\ast}(\Gm),\Z)$ defined by
\[
\lambda \mapsto \mathrm{val}(\lambda(t)),
\]
where $\mathrm{val}$ is the normalized valuation of $E^{\times}$.
Then, since $\det(\varphi_{a})$ is a uniformizer of $E^{\times}$, the image of $\varphi_{a}$ under the Kottwitz homomorphism $\kappa_{\GL_{N,E}}$ generates $X^{\ast}(Z_{\widehat{\GL_{N,E}}})$
\end{proof}

For a character $\omega$ on $\tilde{k}^{\times}$, $a \in \tilde{k}^{\times}$, and $\zeta\in\C^{\times}$, we define a character 
$\chi^{\G_{N}}_{\omega,a,\zeta}$ on the group $Z_{\G_{N}}I_{\G_{N}}^{+}\lan\varphi_{a^{-1}}\ran$ to be
\begin{align*}
 \chi^{\G_{N}}_{\omega,a,\zeta}(z) &= \omega(\ol{z}) \text{ for $z \in Z_{\G_{N}}(q)$,}\\
 \chi^{\G_{N}}_{\omega,a,\zeta}(x) &= \tilde{\psi}\left(\ol{x_{12}}+\cdots+\ol{x_{N-1, N}}+a\ol{x_{N, 1}\varpi^{-1}}\right),\\
 \chi^{\G_{N}}_{\omega,a,\zeta}(\varphi_{a^{-1}}) &= \zeta.
\end{align*}
Then $\chi_{\omega,a,\zeta}^{\G_{N}}|_{Z_{\G_{N}}I_{\G_{N}}^{+}}$ is an affine generic character and we have
\[
N_{\G_{N}(F)}\bigl(I_{\G_{N}}^{+};\chi_{\omega,a,\zeta}^{\G_{N}}|_{Z_{\G_{N}}I_{\G_{N}}^{+}}\bigr)
=
Z_{\G_{N}}I_{\G_{N}}^{+}\lan\varphi_{a^{-1}}\ran
\]

Let $\pi_{\omega,a,\zeta}^{\G_{N}}$ be the simple supercuspidal representation of $\G_{N}(F)$ defined by
\[
\pi_{\omega,a,\zeta}^{\G_{N}}:=\cInd^{\G_{N}(F)}_{Z_{\G_{N}}I_{\G_{N}}^{+}\lan\varphi_{a^{-1}}\ran} \chi^{\G_{N}}_{\omega,a,\zeta}.
\]
Then, for example by using Proposition 2.6 in \cite{Oi:2018}, we can check that the set 
\[
\SSC(\G_{N})
:=
\big\{(\omega,a,\zeta) \mid \omega\colon \tilde{k}^{\times}\ra\C^{\times}, a\in\tilde{k}^{\times}, \zeta\in\C^{\times} \big\}
\]
parametrizes the set of equivalence classes of simple supercuspidal representations of $\G_{N}(F)$.

Now we investigate the $\theta$-stable (conjugate self-dual) simple supercuspidal representations.
Let $\theta$ be an automorphism of $\G_{N}$ defined by
\[
\theta(g)=J_{N} {}^{t}\!c(g)^{-1}J_{N}^{-1}.
\]
Then the automorphism $\theta$ preserves the subgroups $Z_{\G_{N}}$, $I_{\G_{N}}^{+}$, and $I_{\G_{N}}^{++}$, and induces the automorphisms of $Z_{\G_{N}}$ and $V_{\G_{N}} \cong \tilde{k}^{\oplus N}$ as follows:
\[
\theta(z)=c(z)^{-1}
 \quad\text{for}\quad z\in Z_{\G_{N}},
\]
\[
\theta(x_1, \ldots, x_{N-1}, x_{N})=
\bigl(c(x_{N-1}), \ldots, c(x_1), (-1)^{N}c(x_{N})\bigr)
 \quad\text{for}\quad (x_{i})_{i}\in V_{\G_{N}}.
\]
On the other hand, by a simple computation, we can check 
\[
\theta(\varphi_{a^{-1}})=-\varphi_{(-1)^{N}c(a)^{-1}}^{-1}.
\]
Therefore we have
\begin{align*}
(\pi^{\G_{N}}_{\omega,a,\zeta})^{\theta} 
&\cong \cInd_{Z_{\G_{N}}I_{\G_{N}}^{+}\lan\theta(\varphi_{a^{-1}})\ran}^{\G_{N}(F)} (\chi^{\G_{N}}_{\omega,a,\zeta})^{\theta} \\
&\cong \cInd_{Z_{\G_{N}}I_{\G_{N}}^{+}\lan\varphi_{(-1)^{N}c(a)^{-1}}\ran}^{\G_{N}(F)} \chi^{\G_{N}}_{c(\omega)^{-1},(-1)^{N}c(a),\zeta^{-1}\omega(-1)}\\
&= \pi^{\G_{N}}_{c(\omega)^{-1},(-1)^{N}c(a),\zeta^{-1}\omega(-1)}.
\end{align*} 
Thus, the representation $\pi^{\G_{N}}_{\omega,a,\zeta}$ is $\theta$-stable if and only if the following conditions hold:
\begin{enumerate}
\item
$\omega$ is conjugate self-dual, that is $\omega=c(\omega)^{-1}$ (note that this condition is equivalent to the condition that $\omega$ is trivial on $\Nr(\tilde{k}^{\times})=k^{\times}$),
\item
$a=(-1)^{N}c(a)$, and
\item
$\zeta=\zeta^{-1}$ (note that we always have $\omega(-1)=1$ under the condition (1)).
\end{enumerate}
In summary, we can parametrize the set of equivalence classes of $\theta$-stable simple supercuspidal representations of $\G_{N}(F)$ by the following set:
\[
\SSC^{\theta}(\G_{N})
\]
\[
:=
\begin{cases}
\{(\omega,a,\zeta)\in\SSC(\G_{N}) \mid \omega\in(\tilde{k}^{\times}/k^{\times})^{\vee}, a\in k^{\times}, \zeta=\pm1 \}& \text{ if $N$ is even}, \\
\{(\omega,a,\zeta)\in\SSC(\G_{N}) \mid \omega\in(\tilde{k}^{\times}/k^{\times})^{\vee}, a\in (\tilde{k}^{0})^{\times}, \zeta=\pm1 \}& \text{ if $N$ is odd}.
\end{cases}
\]
Here $(\tilde{k}^{\times}/k^{\times})^{\vee}$ is the set of characters on $\tilde{k}^{\times}$ which are trivial on $k^{\times}$ and we put
\[
(\tilde{k}^{0})^{\times}:=\tilde{k}^{0}\cap \tilde{k}^{\times}=\epsilon k^{\times}.
\]

\subsection{Parametrization: the case of $\U_{N}$}
In this subsection, we consider the case of the unramified quasi-split unitary group of degree $N$:
\[
\U_{N} := \{g \in \mathrm{Res}_{E/F}(\GL_{N}) \mid {}^{t}\!c(g)J_{N}g=J_{N}\}.
\]
Note that we can identify the center $\bfZ_{\U_{N}}$ with $\U_{1}$.

The choices and data $\bfS_{\U_{N}}$, $\T_{\U_{N}}$, $\Phi_{\U_{N}}$, $\Psi_{\U_{N}}$, $\Delta_{\U_{N}}$, $\Pi_{\U_{N}}$, $V_{\U_{N}}$, and a set of representatives of equivalence classes of affine generic characters needed to obtain simple supercuspidal representations are described as follows:

\begin{description}
\item[Even case: $N=2n$]
We choose $\bfS_{\U_{2n}}$ to be the diagonal maximal split torus whose $F$-valued points are given by
\[
S_{\U_{2n}}
=
\{\diag(t_1, \ldots, t_n, t_n^{-1}, \ldots, t_1^{-1}) \mid t_i \in F^{\times}\}.
\]
Then its centralizer $\T_{\U_{2n}}$ is a maximal torus with $F$-valued points
\[
T_{\U_{2n}}
=
\{\diag(t_1, \ldots, t_n, c(t_n)^{-1}, \ldots, c(t_1)^{-1}) \mid t_i \in E^{\times}\}.
\]
The relative root system and the affine root system are given by
\[
\Phi_{\U_{2n}}=\{\pm e_i\pm e_j \mid 1\leq i<j\leq n\} \cup \{\pm2e_{i} \mid 1\leq i\leq n\},
\]
\[
\Psi_{\U_{2n}}=\{a+r \mid a\in\Phi_{\U_{2n}}, r\in\Z\}.
\]
We take the root basis and the affine root basis to be 
\[
\Delta_{\U_{2n}}=\{e_1-e_2, \ldots, e_{n-1}-e_n, 2e_n\},
\]
\[
\Pi_{\U_{2n}}=\{e_1-e_2, \ldots, e_{n-1}-e_n, 2e_n, -2e_1+1\}.
\]
We note that, for these simple affine roots, the corresponding affine root subgroups are described as follows:
\begin{itemize}
\item If $\alpha=e_{i}-e_{i+1}$ for $1\leq i\leq n-1$, then
 \[
 U_\alpha=\{I_{2n}+u_{e_{i}-e_{i+1}}(x) \in \U_{2n}(F) \mid x\in\mcO_{E}\}, 
 \] 
 where 
 \[
 \bigl(u_{e_{i}-e_{i+1}}(x)\bigr)_{kl}= 
 \begin{cases}
  x & (k, l)=(i, i+1),\\
  c(x) & (k, l)=(2n-i, 2n-i+1),\\
  0 & $otherwise$.
 \end{cases}
 \]
\item If $\alpha=2e_{n}$, then
 \[
 U_\alpha=\{I_{2n}+u_{e_{2n}}(x) \in \U_{2n}(F) \mid x\in\mcO\}, 
 \]
 where 
 \[
 \bigl(u_{e_{2n}}(x)\bigr)_{kl}=
 \begin{cases}
  x & (k, l)=(n, n+1),\\
  0 & $otherwise$.
 \end{cases}
 \]
\item If $\alpha=-2e_{1}+1$, then
 \[
 U_\alpha=\{I_{2n}+u_{-2e_{1}}(x) \in \U_{2n}(F) \mid x\in\mfp\}, 
 \]
 where 
 \[
 \bigl(u_{-2e_{1}}(x)\bigr)_{kl}=
 \begin{cases}
  x & (k, l)=(2n, 1),\\
  0 & $otherwise$.
 \end{cases}
 \]
\end{itemize}
For $y=(y_{ij})_{ij} \in I_{\U_{2n}}^{+}$, we regard its simple affine components as an element of $\tilde{k}^{\oplus n-1}\oplus k \oplus k$ by
\begin{align*}
I_{\U_{2n}}^{+}/I_{\U_{2n}}^{++} &\cong \bigoplus_{\alpha\in\Pi_{\U_{2n}}}V(\dot{\alpha}) \cong \tilde{k}^{\oplus n-1}\oplus k \oplus k \\
(y_{ij})_{ij} &\mapsto \left(\ol{y_{1 2}}, \ldots, \ol{y_{n-1, n}}, \ol{y_{n, n+1}}, \ol{y_{2n, 1}\varpi^{-1}}\right).
\end{align*}
For $b \in k^{\times}$, we define a character $\chi^{\U_{2n}}_{b} \colon I_{\U_{2n}}^{+} \ra \C^{\times}$ by 
\[
\chi^{\U_{2n}}_{b}(y) := \tilde{\psi}\left(\ol{y_{12}}+\cdots+\ol{y_{n-1, n}}\right) \cdot \psi\left(\ol{y_{n, n+1}}+b\ol{y_{2n, 1}\varpi^{-1}}\right).
\]

\item[Odd case: $N=2n+1$]
We choose $\bfS_{\U_{2n+1}}$ to be the diagonal maximal split torus whose $F$-valued points are given by
\[
S_{\U_{2n+1}} = \big\{t=\diag(t_1, \ldots, t_n, 1, t_n^{-1}, \ldots, t_1^{-1}) \mid t_i \in F^{\times}\big\}.
\]
Then its centralizer $\T_{\U_{2n+1}}$ is a maximal torus with $F$-valued points
\[
T_{\U_{2n+1}} = \big\{t=\diag(t_1, \ldots, t_n, z, c(t_n)^{-1}, \ldots, c(t_1)^{-1}) \mid t_i \in E^{\times}, z\in\U_{1}(F)\big\}.
\]
Then we have 
\[
\Phi_{\U_{2n+1}}=\{\pm e_i\pm e_j \mid 1\leq i<j\leq n\} \cup \{\pm e_{i}, \pm 2e_{i}\mid 1\leq i \leq n\},
\]
\[
\Psi_{\U_{2n+1}}=\{a+r \mid a\in\Phi_{\U_{2n+1}}, r\in\Z\}.
\]
We take the root basis and the affine root basis to be
\[
\Delta_{\U_{2n+1}}=\{e_1-e_2, \ldots, e_{n-1}-e_n, e_n\},
\]
\[
\Pi_{\U_{2n+1}}=\{e_1-e_2, \ldots, e_{n-1}-e_n, e_n, -2e_1+1\}.
\]
We note that, for these simple affine roots, the corresponding affine root subgroups are described as follows:
\begin{itemize}
\item If $\alpha=e_{i}-e_{i+1}$ for $1\leq i \leq n-1$, then
 \[
 U_\alpha=\{I_{2n+1}+u_{e_{i}-e_{i+1}}(x) \in \U_{2n+1}(F) \mid x\in\mcO_{E}\}, 
 \] 
 where 
 \[
 \bigl(u_{e_{i}-e_{i+1}}(x)\bigr)_{kl}= 
 \begin{cases}
  x & (k, l)=(i, i+1),\\
  c(x) & (k, l)=(2n-i+1, 2n-i+2),\\
  0 & $otherwise$.
 \end{cases}
 \]
\item If $\alpha=e_{n}$, then
 \[
 U_\alpha=\{I_{2n+1}+u_{e_{n}}(x, y) \in \U_{2n+1}(F) \mid x, y\in\mcO_{E},\, xc(x)=y+c(y)\}, 
 \]
 where 
 \[
 \bigl(u_{e_{n}}(x,y)\bigr)_{kl}=
 \begin{cases}
  x & (k, l)=(n, n+1),\\
  c(x) & (k, l)=(n+1, n+2),\\
  y & (k, l)=(n, n+2),\\
  0 & $otherwise$.
 \end{cases}
 \]
\item If $\alpha=-2e_{1}+1$, then
 \[
 U_\alpha=\{I_{2n+1}+u_{-2e_{1}}(y) \in \U_{2n+1}(F) \mid y\in\mfp_{E},\, y+c(y)=0\}, 
 \]
 where 
 \[
 \bigl(u_{-2e_{1}}(y)\bigr)_{kl}=
 \begin{cases}
  y & (k, l)=(2n+1, 1),\\
  0 & $otherwise$.
 \end{cases}
 \]
\end{itemize}
For $y=(y_{ij})_{ij} \in I_{\U_{2n+1}}^{+}$, we regard its simple affine components as an element of $\tilde{k}^{\oplus n}\oplus \tilde{k}^{0}$ by
\begin{align*}
I_{\U_{2n+1}}^{+}/I_{\U_{2n+1}}^{++} &\cong \bigoplus_{\alpha\in\Pi_{\U_{2n+1}}}V(\dot{\alpha}) \cong \tilde{k}^{\oplus n}\oplus \tilde{k}^{0} \\
(y_{ij})_{ij} &\mapsto \left(\ol{y_{1 2}}, \ldots, \ol{y_{n, n+1}}, \ol{y_{2n+1, 1}\varpi^{-1}}\right).
\end{align*}
For $b \in (\tilde{k}^{0})^{\times}$, we define a character $\chi^{\U_{2n+1}}_{b} \colon I_{\U_{2n+1}}^{+} \ra \C^{\times}$ by 
\[
\chi^{\U_{2n+1}}_{b}(y) := \tilde{\psi}\left(\ol{y_{12}}+\cdots+\ol{y_{n, n+1}}\right) \cdot \psi\left(b\ol{y_{2n+1, 1}\varpi^{-1}}\right).
\]
\end{description}


\begin{lem}\label{lem:N-Iwahori-U}
We have $N_{\U_{N}(F)}(I_{\U_{N}})=N_{\U_{N}(F)}(I_{\U_{N}}^{+})=I_{\U_{N}}$.
\end{lem}

\begin{proof}
By the same argument as in the proof of Lemma \ref{lem:N-Iwahori}, it is enough to show that the group
\[
X^{\ast}(Z_{\widehat{\U_{N}}})^{\Gal(F^{\ur}/F)}
\]
is trivial.
Here $F^{\ur}$ is the maximal unramified extension of $F$.
Since $X^{\ast}(Z_{\widehat{\U_{N}}}) \cong \Z$ and $\Gal(F^{\ur}/F)$ acts on it by $(-1)$-multiplication, its fixed part is trivial (see Section \ref{sec:endo}.1 for a description of the dual group of $\U_{N}$ and the Galois action on it).
\end{proof}

For a character $\omega'$ on $\tilde{k}^{1}$, we define an affine generic character $\chi^{\U_{N}}_{\omega',b}$ on $Z_{\U_{N}}I_{\U_{N}}^{+}$ by
\begin{align*}
 \chi^{\U_{N}}_{\omega',b}(z)&=\omega'(\ol{z}) \quad\text{for}\quad z\in Z_{\U_{N}}(q), \text{ and}\\
 \chi^{\U_{N}}_{\omega',b}(y)&= \chi^{\U_{N}}_{b}(y) \quad\text{for}\quad y\in I_{\U_{N}}^{+}.
\end{align*} 
Then we have $N_{\U_{N}(F)}(I_{\U_{N}}^{+};\chi^{\U_{N}}_{\omega',b})=Z_{\U_{N}}I_{\U_{N}}^{+}$.
Let $\pi^{\U_{N}}_{\omega',b}$ be the simple supercuspidal representation of $\U_{N}(F)$ defined by
\[
\pi^{\U_{N}}_{\omega',b}:=\cInd^{\U_{N}(F)}_{Z_{\U_{N}}I_{\U_{N}}^{+}} \chi^{\U_{N}}_{\omega',b}.
\]
Then, for example by using Proposition 2.6 in \cite{Oi:2018}, we can check that the set 
\[
\SSC(\U_{N})
:=
\begin{cases}
\big\{(\omega',b) \mid \omega'\in (\tilde{k}^{1})^{\vee}, b\in k^{\times} \big\} & \text{if }N=2n,\\
\big\{(\omega',b) \mid \omega'\in (\tilde{k}^{1})^{\vee}, b\in (\tilde{k}^{0})^{\times} \big\} & \text{if }N=2n+1
\end{cases}
\]
parametrizes the set of equivalence classes of simple supercuspidal representations of $\U_{N}(F)$.
Here we denote the set of characters on $\tilde{k}^{1}$ by $(\tilde{k}^{1})^{\vee}$.

\section{Characters of simple supercuspidal representations}\label{sec:char}
\subsection{Character, character formula, and their twisted versions}
We first recall the (twisted) characters of representations and the character formula.
Let $\bfG$ be a connected reductive group over $F$, and $\pi$ an irreducible representation of $G:=\bfG(F)$.
Then, by the theorem of Harish-Chandra (\cite{MR0414797}), we have its \textit{character} $\Theta_{\pi}$, which is a locally constant function on the set $G^{\rs}$ of regular semisimple elements of $G$.

For irreducible supercuspidal representations which are obtained by the compact induction, we have the following character formula:
\begin{thm}[Character formula, \cite{MR1039842}]\label{thm:CF}
Let $\bfZ_{\bfG}$ be the center of $\bfG$.
Let $K$ be an open subgroup of $G$ such that $K$ contains $Z_{\bfG}$ and $K/Z_{\bfG}$ is compact.
Let $\rho$ be a finite-dimensional irreducible smooth representation of $K$.
We assume that the representation $\pi:=\cInd_K^{G} \rho$ is irreducible, hence supercuspidal.
Then, for every $g \in G^{\rs}$, we have
\[
\Theta_\pi(g)
=\sum_{x\in K\backslash G/K}\sum_{\begin{subarray}{c}y\in K\backslash KxK \\ ygy^{-1} \in K \end{subarray}} \tr\rho(ygy^{-1}).
\]
In particular, we have
\[
\Theta_\pi(g)
=\sum_{\begin{subarray}{c} y\in K \backslash G\\ ygy^{-1} \in K \end{subarray}} \tr\rho(ygy^{-1}),
\]
provided that the sum is finite.
\end{thm}

We next recall the notion of \textit{twisted character} for $\G_{N}(F)$.
Let $\pi$ be an irreducible smooth representation of $\G_{N}(F)$.
We assume that $\pi$ is $\theta$-stable.
Then, by fixing an isomorphism $I_{\theta}\colon\pi\cong\pi^{\theta}$, we get its $\theta$-twisted character $\Theta_{\pi,\theta}$, which is a locally constant function on the set $\G_{N}^{\trs}(F)$ of $\theta$-regular $\theta$-semisimple elements of $\G_{N}(F)$ (see \cite{MR889110} and \cite{MR3632513} for details).
Note that the $\theta$-twisted character $\Theta_{\pi,\theta}$ depends on the choice of an isomorphism $I_{\theta}$.

Similar to the untwisted case, we have a formula for the twisted characters of supercuspidal representations:
\begin{thm}[Twisted character formula, {\cite[Partie I, Th\'eor\`eme 6.2.1]{MR3632513}}]\label{thm:TCF}
Let $K$ be a $\theta$-stable open subgroup of $\G_{N}(F)$ such that $K$ contains $Z_{\G_{N}}$ and $K/Z_{\G_{N}}$ is compact.
Let $\rho$ be a finite-dimensional $\theta$-stable irreducible smooth representation of $K$.
We fix an isomorphism $I_{\theta}\colon\rho\cong\rho^{\theta}$.
We assume that the representation $\pi:=\cInd_K^{\G_{N}(F)} \rho$ is irreducible, hence supercuspidal.
Then, for every $g \in \G_{N}^{\trs}(F)$, we have
\[
\Theta_{\pi, \theta}(g)
= \sum_{x\in K\backslash \G_{N}(F)/K}\sum_{\begin{subarray}{c} y\in K \backslash KxK\\ yg\theta(y)^{-1} \in K \end{subarray}} \tr\bigl(\rho(yg\theta(y)^{-1})\circ I_{\theta}\bigr).
\]
In particular, we have
\[
\Theta_{\pi, \theta}(g) = \sum_{\begin{subarray}{c} y\in K \backslash \G_{N}(F)\\ yg\theta(y)^{-1} \in K \end{subarray}} \tr\bigl(\rho(yg\theta(y)^{-1})\circ I_{\theta}\bigr),
\]
provided that the sum is finite.
Here we normalize the $\theta$-twisted character $\Theta_{\pi, \theta}$ with respect to the isomorphism $\cInd_K^{\G_{N}(F)} I_{\theta} \colon \pi \ra \pi^\theta$.
\end{thm}

In this paper, for a $\theta$-stable simple supercuspidal representation $\pi_{\omega,a,\zeta}^{\G_{N}}$ of $\G_{N}(F)$, we adopt the normalization 
\[
\cInd_{Z_{\G_{N}}I_{\G_{N}}^{+}\lan\varphi_{a^{-1}}\ran}^{\G_{N}(F)} \mathrm{id}\colon\pi_{\omega,a,\zeta}^{\G_{N}}\cong(\pi_{\omega,a,\zeta}^{\G_{N}})^{\theta}
\]
induced from the identity map $\mathrm{id}\colon\chi_{\omega,a,\zeta}^{\G_{N}}\rightarrow(\chi_{\omega,a,\zeta}^{\G_{N}})^{\theta}$.

\begin{rem}\label{rem:normalization}
By using a $\theta$-stable Whittaker data of $\G_{N}$, we can normalize the $\theta$-twisted character of each $\theta$-stable irreducible smooth representation in a natural way depending only on the Whittaker data (see \cite[Section 3.2]{MR3338302}).
In fact, for a simple supercuspidal representation of $\G_{N}(F)$, such a normalization coincides with the above one.
This can be proved by describing the Whittaker normalization for simple supercuspidal representations explicitly from their constructions (see, for example, \cite[Proposition 5.1]{Oi:2016} for the details).
\end{rem}

Before we start to compute the characters of simple supercuspidal representations, we define the Kloosterman sum, which will be used to express the character values.
\begin{defn}\label{defn:Kl}
Let $m$ and $n$ be non-negative integers.
For $u\in k^{\times}$, we put
\[
\Kl_{u}^{m;n} (\psi)
:= \sum_{\begin{subarray}{c} x_1, \ldots , x_m \in k\\ y_1, \ldots , y_n \in \tilde{k}\\ x_1\cdots x_m \Nr(y_1)\cdots\Nr(y_n)= u \end{subarray}} 
\psi(x_1 + \cdots + x_m)
\tilde{\psi}(y_1 + \cdots + y_n).
\] 
\end{defn}

As a consequence of the Hasse--Davenport relation, we have the following lemma:
\begin{lem}\label{lem:HDKl}
For $n \in \Z_{\geq0}$ and $a\in k^{\times}$, we have
\[
\Kl_{a}^{2;n}(\psi)=-\Kl_{a}^{0:n+1}(\psi).
\]
\end{lem}

\begin{proof}
It suffices to show the equality for their Fourier transforms with respect to every multiplicative character of $k^{\times}$.
Let $\chi$ be a multiplicative character of $k^{\times}$.
By, for example, Proposition A.2 in \cite{Oi:2018}, we have
\[
\sum_{a\in k^{\times}} \chi(a) \Kl_{a}^{2;n}(\psi)
=G(k; \chi, \psi)^{2}\cdot G(\tilde{k}; \chi\circ\Nr, \tilde{\psi})^{n}.
\]
On the other hand, we have
\[
-\sum_{a\in k^{\times}} \chi(a) \Kl_{a}^{0;n+1}(\psi)
=-G(\tilde{k}; \chi\circ\Nr, \tilde{\psi})^{n+1}.
\]
Here, $G(k; \chi, \psi)$ (resp.\ $G(\tilde{k}; \chi\circ\Nr, \tilde{\psi})$) is the Gauss sum with respect to $(k,\chi,\psi)$ (resp.\ $(\tilde{k}, \chi\circ\Nr, \tilde{\psi})$).
Therefore we get the conclusion by the Hasse--Davenport relation (see, for example, \cite[503 page]{MR0029393} or \cite[Chapter 11, Theorem 1]{MR1070716}).
\end{proof}

Finally, we introduce a lemma which is useful to determine the index set of the character formula:
\begin{lem}[Lemma 2.5 in \cite{Oi:2018}]\label{lem:key}
Let $\bfG$, $I_{\bfG}$, and $I_{\bfG}^{+}$ be as in Section $\ref{sec:ssc}.1$.
Let $x$ be an element of $G$.
If $x$ satisfies $xgx^{-1}\in I_{\bfG}$ for an affine generic element $g$, then we have $x\in N_{G}(I_{\bfG})$.
\end{lem}

\subsection{The case of twisted $\G_{N}$}
For $g \in \G_{N}(F)$, we put 
\[
\mathcal{N}(g) := g\theta(g) \in \G_{N}(F).
\]

Let $(\omega,a,\zeta)\in\SSC^{\theta}(\G_{N})$, and we consider the corresponding simple supercuspidal representation $\pi^{\G_{N}}_{\omega,a,\zeta}$.
We denote the $\theta$-twisted character $\Theta_{\pi^{\G_{N}}_{\omega,a,\zeta}, \theta}$ of $\pi^{\G_{N}}_{\omega,a,\zeta}$ simply by $\Theta^{\G_{N}}_{\omega,a,\zeta,\theta}$.

First, we compute the twisted characters at $g \in I_{\G_{N}}^{+}\cap \G_{N}^{\trs}(F)$ such that $\mathcal{N}(g)=g\theta(g) \in I_{\G_{N}}^{+}$ is affine generic.

\begin{lem}\label{lem:sumTG}
Let $g \in I_{\G_{N}}^{+}$ be an element such that $\mathcal{N}(g)$ is affine generic.
If $y \in \G_{N}(F)$ satisfies $yg\theta(y)^{-1} \in Z_{\G_{N}}I_{\G_{N}}^{+}\lan\varphi_{a^{-1}}\ran$, then $y \in Z_{\G_{N}}I_{\G_{N}}\lan\varphi_{a^{-1}}\ran$.
\end{lem}

\begin{proof}
Since $yg\theta(y)^{-1} \in Z_{\G_{N}}I_{\G_{N}}^{+}\lan\varphi_{a^{-1}}\ran$, $\mathcal{N}(yg\theta(y)^{-1}) = y\mathcal{N}(g)y^{-1}$ also belongs to $Z_{\G_{N}}I_{\G_{N}}^{+}\lan\varphi_{a^{-1}}\ran$.
Moreover, as we have
 \[
 \mathrm{val}\circ\det(y\mathcal{N}(g)y^{-1})
 =\mathrm{val}\circ\det(\mathcal{N}(g))=0,
 \]
$y\mathcal{N}(g)y^{-1}$ lies in $Z_{\G_{N}}(q)I_{\G_{N}}^{+}\subset I_{\G_{N}}$
(note that we have $\mathrm{val}\circ\det(\varphi_{a^{-1}})=1$).
Then, by the assumption and Lemma \ref{lem:key}, $y$ must lie in $N_{\G_{N}(F)}(I_{\G_{N}})=Z_{\G_{N}}I_{\G_{N}}\langle \varphi_{a^{-1}} \rangle$.
\end{proof}

\begin{lem}\label{lem:sumTG2}
Let $g \in I_{\G_{N}}^{+}$ be an element such that $\mathcal{N}(g)$ is affine generic.
Then a system of representatives of the set 
\[
\big\{ y \in Z_{\G_{N}}I_{\G_{N}}^{+}\lan\varphi_{a^{-1}}\ran\backslash Z_{\G_{N}}I_{\G_{N}}\lan\varphi_{a^{-1}}\ran \mid yg\theta(y)^{-1} \in Z_{\G_{N}}I_{\G_{N}}^{+}\lan\varphi_{a^{-1}}\ran\big\}
\]
is given by 
\[
T'_{\G_{N}}(q)
:= \{ \diag(t_1, \ldots, t_{N})\in T_{\G_{N}}(q) \mid t_1 c(t_{N})=\cdots=t_{N}c(t_1), t_{n}=1\},
\]
where we put $n=\lfloor\frac{N+1}{2}\rfloor$.
\end{lem}

\begin{proof} 
Let $y := \diag(t_1, \ldots, t_{N}) \in T_{\G_{N}}(q)$ satisfying $t_n=1$ (note that the set $Z_{\G_{N}}I_{\G_{N}}^{+}\lan\varphi_{a^{-1}}\ran\backslash Z_{\G_{N}}I_{\G_{N}}\lan\varphi_{a^{-1}}\ran$ is represented by 
\[
\{\diag(t_1, \ldots, t_{N}) \in T_{\G_{N}}(q)\mid t_{n}=1\}).
\]
From the same argument as in the proof of Lemma \ref{lem:sumTG}, if $yg\theta(y)^{-1}$ belongs to $Z_{\G_{N}}I_{\G_{N}}^{+}\lan\varphi_{a^{-1}}\ran$, then it in fact lies in $ Z_{\G_{N}}(q)I_{\G_{N}}^{+}$. 
Thus we have $yg\theta(y)^{-1} \in Z_{\G_{N}}I_{\G_{N}}^{+}\lan\varphi_{a^{-1}}\ran$ if and only if 
the diagonal part of 
\begin{align*}
yg\theta (y)^{-1} &= 
\diag(t_{1}, \ldots, t_{N})
\begin{pmatrix}
g_{1, 1} & \hdots & g_{1, N} \\
\vdots & \ddots & \vdots \\
g_{N, 1} & \hdots & g_{N, N}
\end{pmatrix}
\diag\bigl(c(t_{N}), \ldots, c(t_{1})\bigr)
\\
&= 
\begin{pmatrix}
 t_1c(t_{N})g_{1, 1}&t_1c(t_{N-1})g_{1, 2}&&\\
 &t_2c(t_{N-1})g_{2, 2}&\ddots&\ast\\
 &\ast&\ddots&t_{N-1}c(t_1)g_{N-1, N}\\
 t_{N}c(t_{N})g_{N, 1}&&&t_{N}c(t_1)g_{N, N}
\end{pmatrix}
\end{align*}
lies in $Z_{\G_{N}}(q)T_{\G_{N}}^{1}$, and this is equivalent to that $t_1c(t_{N})=\cdots=t_{N}c(t_1)$.
\end{proof}

\begin{prop}[Even case: $N=2n$]\label{prop:charTG}
Let $g \in I_{\G_{2n}}^{+}\cap \G_{2n}^{\trs}(F)$ be an element such that $\mathcal{N}(g)$ is affine generic.
Let $(g_1, \ldots, g_{2n})$ be the simple affine components of $g$.
Then we have 
\[
\Theta^{\G_{2n}}_{\omega,1,\zeta,\theta}(g) = 
-\Kl^{0;n}_{\Nr(g_1+c(g_{2n-1}))\cdots \Nr(g_{n-1}+c(g_{n+1}))\Tr(g_{n})\Tr(g_{2n})}(\psi).
\]
\end{prop}

\begin{proof}
Note that the simple affine components of $\mathcal{N}(g)$ are given by 
\[
\bigl(g_1+c(g_{2n-1}), \ldots, g_{2n-1}+c(g_1), \Tr(g_{2n})\bigr),
\]
and the affine genericity of $\mathcal{N}(g)$ means that none of them is zero.

From Lemma \ref{lem:sumTG}, the index set of the twisted character formula (Theorem \ref{thm:TCF}) is finite, and given by $T'_{\G_{2n}}(q)$ from Lemma \ref{lem:sumTG2}.
Thus, by noting that the element $yg\theta(y)^{-1}$ belongs to 
\[
\diag(\alpha,\ldots,\alpha)I_{\G_{2n}^{+}},
\]
where $\alpha=t_1 c(t_{2n})=\cdots=t_{2n}c(t_1)\in k^{\times}$, and that the character $\omega$ is trivial on $k^{\times}$, we can compute the twisted character as follows: 
\begin{align*}
&\Theta^{\G_{2n}}_{\omega,1,\zeta,\theta}(g)
= \sum_{y \in T'_{\G_{2n}}(q)} \chi^{\G_{2n}}_{\omega,1,\zeta}\bigl(yg\theta(y)^{-1}\bigr) \\
&= \sum_{\alpha \in k^{\times}} \omega(\alpha) \sum_{\begin{subarray}{c} t_1, \ldots, t_{2n} \in \tilde{k}^{\times}\\ t_i c(t_{2n+1-i}) =\alpha \\ t_n=1 \end{subarray}} 
\tilde{\psi}\left(\frac{t_1 c(t_{2n-1}) g_1}{\alpha} + \cdots + \frac{t_{2n-1} c(t_1) g_{2n-1}}{\alpha} + \frac{\Nr(t_{2n}) g_{2n}}{\alpha}\right) \\
&=\sum_{\alpha \in k^{\times}} \sum_{t_1, \ldots, t_{n-1} \in \tilde{k}^{\times}}
\tilde{\psi}\left(\frac{t_1}{t_2}g_1 + \cdots + c\left(\frac{t_1}{t_2}\right)g_{2n-1} + \frac{\alpha}{\Nr(t_1)}g_{2n}\right) \\
&= \sum_{\begin{subarray}{c}t_1, \ldots, t_{n-1} \in \tilde{k}^{\times}\\ \alpha\in k^{\times}\end{subarray}} 
\tilde{\psi} \left( \frac{t_1}{t_2}\bigl(g_1+c(g_{2n-1})\bigr) + \cdots +\frac{t_{n-1}}{1}\bigl(g_{n-1}+c(g_{n+1})\bigr) + \frac{g_n}{\alpha} + \frac{\alpha}{\Nr(t_1)}g_{2n} \right) \\
&=\Kl^{2;n-1}_{\Nr(g_1+c(g_{2n-1}))\cdots \Nr(g_{n-1}+c(g_{n+1}))\Tr(g_{n})\Tr(g_{2n})}(\psi)\\
&=-\Kl^{0;n}_{\Nr(g_1+c(g_{2n-1}))\cdots \Nr(g_{n-1}+c(g_{n+1}))\Tr(g_{n})\Tr(g_{2n})}(\psi).
\end{align*}
Here, we used Lemma \ref{lem:HDKl} in the last equality.
\end{proof}

\begin{prop}[Odd case: $N=2n+1$]\label{prop:charTG_{o}}
Let $g \in I_{\G_{2n+1}}^{+}\cap \G_{2n+1}^{\trs}(F)$ be an element such that $\mathcal{N}(g)$ is affine generic.
Let $(g_1, \ldots, g_{2n+1})$ be the simple affine components of $g$.
Then we have 
\[
\Theta^{\G_{2n+1}}_{\omega,\epsilon,\zeta,\theta}(g) = 
\Kl^{1;n}_{\Nr(g_1+c(g_{2n}))\cdots \Nr(g_{n}+c(g_{n+1}))\Tr(\epsilon g_{2n+1})}(\psi).
\]
\end{prop}

\begin{proof}
Note that the simple affine components of $\mathcal{N}(g)$ are given by 
\[
\bigl(g_1+c(g_{2n}), \ldots, g_{2n}+c(g_1), g_{2n+1}-c(g_{2n+1})\bigr),
\]
and the affine genericity of $\mathcal{N}(g)$ means that none of them is zero.

From Lemma \ref{lem:sumTG}, the index set of the twisted character formula (Theorem \ref{thm:TCF}) is finite, and given by $T'_{\G_{2n+1}}(q)$ from Lemma \ref{lem:sumTG2}.
Thus we can compute the twisted character as follows: 
\begin{align*}
&\Theta^{\G_{2n+1}}_{\omega,\epsilon,\zeta,\theta}(g)
= \sum_{y \in T'_{\G_{2n+1}}(q)} \chi^{\G_{2n+1}}_{\omega,\epsilon,\zeta}\bigl(yg\theta(y)^{-1}\bigr) \\
&= \sum_{\begin{subarray}{c} t_1, \ldots, t_{2n+1} \in \tilde{k}^{\times}\\ t_i c(t_{2n+2-i}) =1\\ t_{n+1}=1\end{subarray}} 
\tilde{\psi}\bigl(t_1 c(t_{2n}) g_1 + \cdots + t_{2n} c(t_1) g_{2n} + \epsilon \Nr(t_{2n+1}) g_{2n+1}\bigr) \\
&= \sum_{t_1, \ldots, t_{n} \in \tilde{k}^{\times}}
\tilde{\psi}\left(\frac{t_1}{t_2}g_1 + \cdots + c\left(\frac{t_1}{t_2}\right)g_{2n} + \frac{\epsilon}{\Nr(t_1)}g_{2n+1}\right) \\
&= \sum_{t_1, \ldots, t_{n} \in \tilde{k}^{\times}}\tilde{\psi} \left( \frac{t_1}{t_2}\bigl(g_1+c(g_{2n})\bigr) + \cdots +\frac{t_{n}}{1}\bigl(g_{n}+c(g_{n+1})\bigr) + \frac{\epsilon}{\Nr(t_1)}g_{2n+1} \right) \\
&=\Kl^{1;n}_{\Nr(g_1+c(g_{2n}))\cdots \Nr(g_{n}+c(g_{n+1}))\Tr(\epsilon g_{2n+1})}(\psi).
\end{align*}
\end{proof}

Next, we compute the twisted characters $\Theta^{\G_{N}}_{\omega,a,\zeta,\theta}$ at $\varphi_{a^{-1}u}g$, where $g \in I_{\G_{N}}^{+}$ and $u \in k^{\times}$ such that $-\mathcal{N}(\varphi_{a^{-1}u} g)=\varphi_{a^{-1}u} g\varphi_{a^{-1}u}^{-1}\theta(g) \in I_{\G_{N}}^{+}$ is affine generic.

\begin{lem}\label{lem:sumTG'}
Let $u\in k^{\times}$.
Let $g \in I_{\G_{N}}^{+}$ be an element such that $-\mathcal{N}(\varphi_{a^{-1}u} g) \in I_{\G_{N}}^{+}$ is affine generic.
If $y \in \G_{N}(F)$ satisfies $y\varphi_{a^{-1}u} g\theta(y)^{-1} \in Z_{\G_{N}}I_{\G_{N}}^{+}\lan\varphi_{a^{-1}}\ran$, then $y \in Z_{\G_{N}}I_{\G_{N}}\lan\varphi_{a^{-1}}\ran$.
\end{lem}

\begin{proof}
This follows from the exactly same argument as in the proof of Lemma \ref{lem:sumTG}.
\end{proof}

\begin{lem}\label{lem:sumTG'2}
Let $u\in k^{\times}$.
Let $g \in I_{\G_{N}}^{+}$ be an element such that $-\mathcal{N}(\varphi_{a^{-1}u} g) \in I_{\G_{N}}^{+}$ is affine generic.
Then a system of representatives of the set 
\[
\big\{ y \in Z_{\G_{N}}I_{\G_{N}}^{+}\lan\varphi_{a^{-1}}\ran\backslash Z_{\G_{N}}I_{\G_{N}}\lan\varphi_{a^{-1}}\ran \mid y\varphi_{a^{-1}u} g\theta (y)^{-1} \in Z_{\G_{N}}I_{\G_{N}}^{+}\lan\varphi_{a^{-1}}\ran\big\}
\]
is given by 
\[
T''_{\G_{N}}(q)
:= \{ \diag(t_1, \ldots, t_{N})\in T_{\G_{N}}(q) \mid t_1c(t_{N-1})= \cdots = t_{N-1}c(t_1) = u, t_{N}=1 \}.
\]
\end{lem}

\begin{proof}
Let $y = \diag(t_1, \ldots, t_{N}) \in T_{\G_{N}}(q)$ satisfying $t_{N}=1$.
Since 
\[
\mathrm{val}\circ\det\bigl(y\varphi_{a^{-1}u} g\theta (y)^{-1}\bigr)=\mathrm{val}\circ\det(\varphi_{a^{-1}u})=1,
\]
we have
\[
y\varphi_{a^{-1}u} g\theta (y)^{-1} \in Z_{\G_{N}}I_{\G_{N}}^{+}\lan\varphi_{a^{-1}}\ran 
\implies
\varphi_{a^{-1}}^{-1}y\varphi_{a^{-1}u} g\theta (y)^{-1} \in Z_{\G_{N}}(q)I_{\G_{N}}^{+}.
\] 
Thus $y\varphi_{a^{-1}u} g\theta (y)^{-1}$ belongs to $Z_{\G_{N}}I_{\G_{N}}^{+}\lan\varphi_{a^{-1}}\ran$ if and only if the diagonal part of
\[
\varphi_{a^{-1}}^{-1}y\varphi_{a^{-1}u} \cdot g \cdot \theta (y)^{-1} 
\]
\begin{align*}
&=
\diag(t_{N}u, t_1, \ldots, t_{N-1})
\begin{pmatrix}
g_{1, 1}&\hdots&g_{1, N}\\
\vdots&\ddots&\vdots\\
g_{N, 1}&\hdots&g_{N, N}
\end{pmatrix}
\diag\bigl(c(t_{N}), \ldots, c(t_{1})\bigr)
\\
&= 
\begin{pmatrix}
 t_{N}c(t_{N})ug_{1, 1}&t_{N}c(t_{N-1})ug_{1, 2}&&\\
 &t_1c(t_{N-1})g_{2, 2}&\ddots&\ast\\
 &\ast&\ddots&t_{N-2}c(t_1)g_{N-1, N}\\
 t_{N-1}c(t_{N})g_{N, 1}&&&t_{N-1}c(t_1)g_{N, N}
\end{pmatrix}
\end{align*}
lies in $Z_{\G_{N}}(q)T_{\G_{N}}^{1}$, and this is equivalent to that $t_1c(t_{N-1})= \cdots = t_{N-1}c(t_1)= u$.
\end{proof}

\begin{prop}[Even case: $N=2n$]\label{prop:charTG'}
Let $g \in I_{\G_{2n}}^{+}$ be an element such that $\varphi_ug \in \G_{2n}^{\trs}(F)$ and $-\mathcal{N}(\varphi_u g)$ is affine generic.
Let $(g_1, \ldots, g_{2n})$ be the simple affine components of $g$.
Then we have 
\[
\Theta^{\G_{2n}}_{\omega,1,\zeta,\theta} (\varphi_u g) =
 \zeta\cdot \Kl^{0;n}_{\Nr(ug_1+c(g_{2n}))\Nr(g_2+c(g_{2n-1}))\cdots\Nr(g_{n}+c(g_{n+1}))/u}(\psi).
\]
\end{prop}

\begin{proof}
Note that the simple affine components of $-\mathcal{N}(\varphi_{u}g)$ are given by 
\[
\bigl(g_2+c(g_{2n-1}), g_3+c(g_{2n-2}), \ldots, g_{2n-1}+c(g_2), u^{-1}g_{2n}+c(g_1), ug_1+c(g_{2n})\bigr),
\]
and the affine genericity of $-N(\varphi_{u}g)$ means that none of them is zero.

From Lemma \ref{lem:sumTG'}, the index set of the twisted character formula (Theorem \ref{thm:TCF}) is finite, and given by $T''_{\G_{2n}}(q)$ from Lemma \ref{lem:sumTG'2}.
Thus, by noting that the element $\varphi_{1}^{-1}y\varphi_{u} g\theta(y)^{-1}$ belongs to 
\[
\diag(u,\ldots,u)I_{\G_{2n}^{+}}
\]
and that the character $\omega$ is trivial on $k^{\times}$, we can compute the twisted character as follows: 
\begin{align*}
& \Theta^{\G_{2n}}_{\omega,1,\zeta,\theta} (\varphi_u g)
= \sum_{y \in T''_{\G_{2n}}(q)} \chi^{\G_{2n}}_{\omega,1,\zeta}(\varphi_{1})\cdot\chi^{\G_{2n}}_{\omega,1,\zeta}\bigl(\varphi_{1}^{-1}y\varphi_{u} g\theta(y)^{-1}\bigr) \\
&= \zeta\omega(u) \sum_{\begin{subarray}{c} t_1, \ldots, t_{2n} \in \tilde{k}^{\times}\\ t_i c(t_{2n-i}) =u\\ t_{2n} = 1 \end{subarray}} 
\tilde{\psi} \left( \frac{t_{2n}c(t_{2n-1})ug_1}{u} + \cdots + \frac{t_{2n-2}c(t_1)g_{2n-1}}{u} + \frac{t_{2n-1}c(t_{2n})g_{2n}}{u} \right) \\
&= \zeta \sum_{\begin{subarray}{c} t_1, \ldots, t_{n} \in \tilde{k}^{\times}\\ \Nr(t_n)=u\end{subarray}} 
\tilde{\psi} \left( \frac{u}{t_1}g_1 + \cdots + c\biggl(\frac{t_1}{t_2}\biggr)g_{2n-1} + \frac{g_{2n}}{c(t_1)}\right) \\
&= \zeta \sum_{\begin{subarray}{c} t_1, \ldots, t_{n} \in \tilde{k}^{\times}\\ \Nr(t_n)=u\end{subarray}} 
\tilde{\psi} \left( \frac{1}{t_1}\bigl(ug_1+c(g_{2n})\bigr) + \frac{t_1}{t_2}\bigl(g_2+c(g_{2n-1})\bigr) + \cdots + \frac{t_{n-1}}{t_n}\bigl(g_{n}+c(g_{n+1})\bigr)\right) \\
&= \zeta\cdot \Kl^{0;n}_{\Nr(ug_1+c(g_{2n}))\Nr(g_2+c(g_{2n-1}))\cdots\Nr(g_{n}+c(g_{n+1}))/u}(\psi).
\end{align*}
\end{proof}

\begin{prop}[Odd case: $N=2n+1$]\label{prop:charTG'_{o}}
Let $g \in I_{\G_{2n+1}}^{+}$ be an element such that $\varphi_{\epsilon^{-1}u}g \in \G_{2n+1}^{\trs}(F)$ and $-\mathcal{N}(\varphi_{\epsilon^{-1}u} g)$ is affine generic.
Let $(g_1, \ldots, g_{2n+1})$ be the simple affine components of $g$.
Then we have 
\[
\Theta^{\G_{2n+1}}_{\omega,\epsilon,\zeta,\theta} (\varphi_{\epsilon^{-1} u} g) = \zeta \cdot \Kl^{1;n}_{\Nr(ug_1-\epsilon c(g_{2n+1}))\Nr(g_2+c(g_{2n}))\cdots\Nr(g_{n}+c(g_{n+2}))\Tr(g_{n+1})/u}(\psi).
\]
\end{prop}

\begin{proof}
Note that the simple affine components of $-\mathcal{N}(\varphi_{\epsilon^{-1}u}g)$ are given by 
\[
\bigl(g_2+c(g_{2n}), g_3+c(g_{2n-1}), \ldots, g_{2n}+c(g_2), \epsilon u^{-1}g_{2n+1}+c(g_1), \epsilon^{-1}ug_1-c(g_{2n+1})\bigr),
\]
and the affine genericity of $\mathcal{N}(g)$ means that none of them is zero.

From Lemma \ref{lem:sumTG'}, the index set of the twisted character formula (Theorem \ref{thm:TCF}) is finite, and given by $T''_{\G_{2n+1}}(q)$ from Lemma \ref{lem:sumTG'2}.
Thus, by noting that the element $\varphi_{\epsilon^{-1}}^{-1}y\varphi_{\epsilon^{-1}u} g\theta(y)^{-1}$ belongs to 
\[
\diag(u,\ldots,u)I_{\G_{2n+1}^{+}}
\]
and that the character $\omega$ is trivial on $k^{\times}$, we can compute the twisted character as follows: 
\begin{align*}
& \Theta^{\G_{2n+1}}_{\omega,\epsilon,\zeta,\theta}(\varphi_{\epsilon^{-1}u} g) 
= \sum_{y \in T''_{\G_{2n+1}}(q)} \chi^{\G_{2n+1}}_{\omega,\epsilon,\zeta}(\varphi_{\epsilon^{-1}})\cdot\chi^{\G_{2n+1}}_{\omega,\epsilon,\zeta}\bigl(\varphi_{\epsilon^{-1}}^{-1}y\varphi_{\epsilon^{-1}u} g\theta(y)^{-1}\bigr) \\
&= \zeta\omega(u) \sum_{\begin{subarray}{c} t_1, \ldots, t_{2n+1} \in \tilde{k}^{\times}\\ t_i c(t_{2n+1-i}) =u\\ t_{2n+1} = 1 \end{subarray}} 
\tilde{\psi} \left( \frac{t_{2n+1}c(t_{2n})ug_1}{u} + \cdots + \frac{t_{2n-1}c(t_1)g_{2n}}{u} + \epsilon\frac{t_{2n}c(t_{2n+1})g_{2n+1}}{u} \right) \\
&= \zeta \sum_{t_1, \ldots, t_{n} \in \tilde{k}^{\times}}
\tilde{\psi} \left( \frac{u}{t_1}g_1 + \cdots + c\left(\frac{t_1}{t_2}\right)g_{2n} + \epsilon\frac{g_{2n+1}}{c(t_1)}\right) \\
&= \zeta \sum_{t_1, \ldots, t_{n} \in \tilde{k}^{\times}}
\tilde{\psi} \left( \frac{1}{t_1}\bigl(ug_1-\epsilon c(g_{2n+1})\bigr) + \frac{t_1}{t_2}\bigl(g_2+c(g_{2n})\bigr) + \cdots + \frac{\Nr(t_{n})}{u}g_{n+1}\right) \\
&=\zeta\cdot \Kl^{1;n}_{\Nr(ug_1-\epsilon c(g_{2n+1}))\Nr(g_2+c(g_{2n}))\cdots\Nr(g_{n}+c(g_{n+2}))\Tr(g_{n+1})/u}(\psi).
\end{align*}
\end{proof}

\subsection{The case of $\U_{N}$}
Let $(\omega',b)\in\SSC(\U_{N})$, and we consider the corresponding simple supercuspidal representation $\pi^{\U_{N}}_{\omega',b}$.

We compute the character of $\pi^{\U_{N}}_{\omega',b}$ at an affine generic element $h \in I_{\U_{N}}^{+}\cap \U_{N}^{\rs}(F)$.

\begin{lem}\label{lem:sumU}
Let $h \in I_{\U_{N}}^{+}$ be an affine generic element.
If $y \in \U_{N}(F)$ satisfies $yhy^{-1} \in Z_{\U_{N}}I_{\U_{N}}^{+}$, then $y \in I_{\U_{N}}$.
\end{lem}

\begin{proof}
If $yhy^{-1}$ belongs to $Z_{\U_{N}}I_{\U_{N}}^{+} \subset I_{\U_{N}}$, then $y$ lies in $N_{\U_{N}(F)}(I_{\U_{N}})=I_{\U_{N}}$ by Lemma \ref{lem:key}.
\end{proof}

\begin{prop}[Even case: $N=2n$]\label{prop:charU}
Let $h \in I_{\U_{2n}}^{+}\cap \U_{2n}^{\rs}(F)$ be an affine generic element with its simple affine components $(h_1, \ldots, h_{n-1}, h_n, h_{0})$.
Then we have 
\[
\Theta_{\omega',b}^{\U_{2n}}(h)
= -\Kl^{0;n}_{\Nr(h_1)\cdots \Nr(h_{n-1})h_nh_{0}b}(\psi).
\]
\end{prop}

\begin{proof}
We fix a set $[\tilde{k}^{\times}/\tilde{k}^{1}]\subset \tilde{k}^{\times}$ of representatives of $\tilde{k}^{\times}/\tilde{k}^{1}$.
Then, by Lemma \ref{lem:sumU}, the index set of the character formula (Theorem \ref{thm:CF}) is finite and represented by
\[
T'_{\U_{2n}}(q):=
\{\diag(t_{1},\ldots,t_{n},c(t_{n})^{-1},\ldots,c(t_{1})^{-1})\mid t_{1},\ldots,t_{n-1}\in\tilde{k}^{\times}, t_{n}\in[\tilde{k}^{\times}/\tilde{k}^{1}]\}.
\]
Therefore, by noting that we have
\[
\tilde{k}^{\times}/\tilde{k}^{1}\cong k;\quad z\mapsto \Nr(z)=zc(z),
\]
we can compute the character as follows:
\begin{align*}
&\Theta_{\omega',b}^{\U_{2n}}(h) = \sum_{t \in T'_{\U_{2n}}(q)} \chi_{\omega',b}^{\U_{2n}}(tht^{-1}) \\
&= \sum_{\begin{subarray}{c}t_1, \ldots, t_{n-1} \in \tilde{k}^{\times} \\ t_n \in [\tilde{k}^{\times}/\tilde{k}^{1}] \end{subarray}} \tilde{\psi} \left( \frac{t_1}{t_2}h_1 + \cdots + \frac{t_{n-1}}{t_n}h_{n-1}\right) \cdot \psi\left(\Nr(t_n)h_n + \frac{b}{\Nr(t_1)}h_{0} \right) \\
&= \sum_{\begin{subarray}{c}t_1, \ldots, t_{n-1} \in \tilde{k}^{\times} \\ t_n \in [\tilde{k}^{\times}/\tilde{k}^{1}] \end{subarray}} \tilde{\psi} \left( \frac{t_1}{t_2}h_1 + \cdots + \frac{t_{n-1}}{1}h_{n-1}\right) \cdot \psi\left(\Nr(t_n)h_n + \frac{b}{\Nr(t_{1}t_{n})}h_{0} \right) \\
&= \sum_{\begin{subarray}{c}t_1, \ldots, t_{n-1} \in \tilde{k}^{\times} \\ s_n \in k^{\times} \end{subarray}} \tilde{\psi} \left( \frac{t_1}{t_2}h_1 + \cdots + \frac{t_{n-1}}{1}h_{n-1}\right) \cdot \psi\left(s_nh_n + \frac{b}{\Nr(t_1) s_n}h_{0} \right)\\
&=\Kl^{2;n-1}_{\Nr(h_1)\cdots \Nr(h_{n-1})h_nh_{0}b}(\psi).
\end{align*}
Here, we replaced $t_{i}$ with $t_{i}t_{n}$ in the third equality.
Finally, from Lemma \ref{lem:HDKl}, the right-hand side of the above equalities is equal to
\[
-\Kl^{0;n}_{\Nr(h_1)\cdots \Nr(h_{n-1})h_nh_{0}b}(\psi).
\]
\end{proof}

\begin{prop}[Odd case: $N=2n+1$]\label{prop:charU_{o}}
Let $h \in I_{\U_{2n+1}}^{+}\cap \U_{2n+1}^{\rs}(F)$ be an affine generic element with its simple affine components $(h_1, \ldots, h_n, h_{0})$.
Then we have 
\[
\Theta^{\U_{2n+1}}_{\omega',b}(h)= \Kl^{1;n}_{\Nr(h_1)\cdots \Nr(h_{n})h_{0}b}(\psi).
\]
\end{prop}

\begin{proof}
From Lemma \ref{lem:sumU}, the index set of the character formula (Theorem \ref{thm:CF}) is finite and represented by
\[
T'_{\U_{2n+1}}(q):=
\{\diag(t_{1},\ldots,t_{n},1,c(t_{n})^{-1},\ldots,c(t_{1})^{-1})\mid t_{1},\ldots,t_{n}\in\tilde{k}^{\times}\}.
\]
Thus we get
\begin{align*}
\Theta^{\U_{2n+1}}_{\omega',b}(h) 
&= \sum_{t \in T'_{\U_{2n+1}}(q)} \chi_{\omega',b}^{\U_{2n+1}}(tht^{-1}) \\
&= \sum_{t_1, \ldots, t_{n} \in \tilde{k}^{\times}} \tilde{\psi} \left( \frac{t_1}{t_2}h_1 + \cdots + \frac{t_{n}}{1}h_{n}\right) \cdot \psi\left(\frac{b}{\Nr(t_1)}h_{0} \right) \\
&=\Kl^{1;n}_{\Nr(h_{1})\cdots \Nr(h_{n})h_{0}b}(\psi).
\end{align*}
\end{proof}

\section{Twisted endoscopy for $\G_{N}$}\label{sec:endo}

\subsection{Simple endoscopic data}
Recall that we set 
\[
\G_{N}=\mathrm{Res}_{E/F}\GL_{N},
\]
and define the automorphism $\theta$ of $G$ over $F$ by
\[
\theta(g)=J_{N}{}^{t}\!c(g)^{-1}J_{N}^{-1}.
\]
Then we have
\[
\widehat{\G_{N}}=\GL_{N}(\C)\times\GL_{N}(\C),
\]
and the Weil group $W_F$ and the dual automorphism $\hat{\theta}$ to $\theta$ act on $\widehat{\G_{N}}$ by the following way:
\begin{itemize}
 \item For $w \in W_F$ and $(g, h) \in \GL_{N}(\C)\times\GL_{N}(\C)$,
 \[
 w(g, h) = 
 \begin{cases}
  (g, h) & \text{for }w \in W_E, \\
  (h, g) & \text{for }w \in W_F\setminus W_E.
 \end{cases}
 \]
 \item For $(g, h)\in\GL_{N}(\C)\times\GL_{N}(\C)$,
 \[
 \hat{\theta}(g, h)=(J_{N}{}^th^{-1}J_{N}^{-1}, J_{N}{}^{t}\!g^{-1}J_{N}^{-1}).
 \]
\end{itemize}

On the other hand, we have
\[
\widehat{\U_{N}}= \GL_{N}(\C),
\]
and an element $w\in W_F$ acts on $g\in\widehat{\U_{N}}$ by
\[
w(g) = 
\begin{cases}
 g & \text{for }w \in W_E, \\
 J_{N}{}^{t}\!g^{-1}J_{N}^{-1} & \text{for }w \in W_F\setminus W_E.
\end{cases}
\]

In this paper, we consider the endoscopy of the following two types:
\begin{description}
 \item[(1) The standard base change embedding]
 We define an $L$-embedding
 \[
 \xi_{+1} \colon {}^L\U_{N} \hookrightarrow {}^L\G_{N}
 \]
 as follows:
 \begin{align*}
  g\rtimes1 &\mapsto (g, J_{N}{}^{t}\!g^{-1}J_{N}^{-1})\rtimes1, \\
  1\rtimes w &\mapsto (I_{N}, I_{N})\rtimes w, \quad\text{for } w\in W_F.
 \end{align*}
 Then $(\U_{N}, {}^L\U_{N}, s=1, \xi_{+1})$ are endoscopic data for the triplet $(\G_{N}, \theta, 1)$.

 \item[(2) The twisted base change embedding]
 We fix $w_c \in W_F\setminus W_E$.
 We define a character $\chi$ on $E^{\times}$ by
 \[
 \chi(x):=(-1)^{\mathrm{val}(x)}.
 \]
 Then this character is unramified, and $\chi|_{F^{\times}}$ is quadratic.
 We regard $\chi$ also as a character of $W_E$ via the local class field theory.
 By using this character, we define an $L$-embedding
 \[
 \xi_{-1} \colon {}^L\U_{N} \hookrightarrow {}^L\G_{N}
 \]
 as follows:
 \begin{align*}
  g\rtimes1 &\mapsto (g, J_{N}{}^t\!g^{-1}J_{N}^{-1})\rtimes1, \\
  1\rtimes \sigma &\mapsto (\chi(\sigma)I_{N}, \chi(\sigma)^{-1}I_{N})\rtimes \sigma \text{, for } \sigma\in W_E,\\
  1\rtimes w_c &\mapsto (-I_{N}, I_{N})\rtimes w_c.
 \end{align*}
 Then $(\U_{N}, {}^L\U_{N}, s=1, \xi_{-1})$ are endoscopic data for the triplet $(\G_{N}, \theta, 1)$.
\end{description}
In fact, there are only these two kinds of simple (in the sense of Arthur) endoscopic data for $\G_{N}$ up to equivalence (see Section 4.7 in \cite{MR1081540}).

\subsection{Norm correspondences}
We recall the norm correspondence for twisted endoscopy in \cite{MR1687096}.
From the endoscopic data defined in the previous subsection, we get a map 
\[
\mathcal{A}_{\U_{N}/\G_{N}} \colon Cl_\mathrm{ss}(\U_{N}) \ra Cl_{\theta\mathchar`-\mathrm{ss}}(\G_{N},\theta) 
\]
from the set of semisimple conjugacy classes in $\U_{N}(\ol{F})$ to the set of $\theta$-semisimple $\theta$-conjugacy classes in $\G_{N}(\ol{F})$ (see Section 3.3 in \cite{MR1687096}).
Here note that this map is made from the restriction of the $L$-embedding $\xi_{\kappa}$ to $\widehat{\U_{N}}$, hence does not depend on $\kappa$.

Let $\G_{N}^{\strs}(F)$ be the set of strongly $\theta$-regular $\theta$-semisimple elements in $\G_{N}(F)$, and $\U_{N}^{\srs}(F)$ the set of strongly regular semisimple elements in $\U_{N}(F)$.
We say that $y\in \U_{N}^{\srs}(F)$ is a $\mathit{norm}$ of $x\in \G_{N}^{\strs}(F)$ if $x$ corresponds to $y$ via the map $\mathcal{A}_{\U_{N}/\G_{N}}$.
Let us recall that this map is described explicitly in terms of the diagonal maximal tori considered in Sections \ref{sec:ssc}.2 and \ref{sec:ssc}.3.
We fix the isomorphisms as follows:
\begin{align*}
 \G_{N}(F)\otimes_{F}\ol{F} &\cong \GL_{N}(\ol{F})\times\GL_{N}(\ol{F})\\
 g\otimes a&\mapsto \bigl(ga, c(g)a\bigr), \text{ and}\\
 \U_{N}(F)\otimes_{F}\ol{F} &\cong \GL_{N}(\ol{F})\\
 h\otimes b&\mapsto hb.
\end{align*}
Then the groups $\T_{\G_{N}}(\ol{F})$ and $\T_{\U_{N}}(\ol{F})$ are identified with subgroups of diagonal matrices in $\GL_{N}(\ol{F})\times\GL_{N}(\ol{F})$ and $\GL_{N}(\ol{F})$ under the above isomorphisms, respectively.
In the right-hand side of the first isomorphism, the involution $\theta$ is described as follows:
\[
\theta(g_{1},g_{2})=(J_{N}{}^{t}\!g_{2}^{-1}J_{N}^{-1}, J_{N}{}^{t}\!g_{1}^{-1}J_{N}^{-1}).
\]
In particular, the action of $\theta$ on $\T_{\G_{N}}$ is given by
\[
\theta
\bigl(\diag(t_1, \ldots, t_{N}), \diag(s_1, \ldots, s_{N})\bigr)
=
\bigl(\diag(s_{N}^{-1}, \ldots, s_{1}^{-1}), \diag(t_{N}^{-1}, \ldots, t_{1}^{-1})\bigr),
\]
and we have
\begin{align*}
\mathcal{A}_{\U_{N}/\G_{N}}\colon
Cl_\mathrm{ss}(\U_{N})
\cong
\T_{\U_{N}}(\ol{F})/\Omega_{\T_{\U_{N}}}
&\cong \T_{\G_{N},\theta}(\ol{F})/\Omega_{\T_{\G_{N}}}^{\theta}
\cong
Cl_{\theta\mathrm{\mathchar`-ss}}(\G_{N}, \theta)\\
\diag\left(\frac{t_1}{s_{N}}, \ldots, \frac{t_{N}}{s_1}\right) &\leftrightarrow \bigl(\diag(t_1, \ldots, t_{N}), \diag(s_1, \ldots, s_{N})\bigr),
\end{align*}
where $\Omega_{\T_{\U_{N}}}$ is the Weyl group of $\T_{\U_{N}}$ in $\U_{N}$, $\Omega_{\T_{\G_{N}}}^{\theta}$ is the $\theta$-fixed part of the Weyl group of $\T_{\G_{N}}$ in $\G_{N}$, and $\T_{\G_{N},\theta}$ is the $\theta$-coinvariant of $\T_{\G_{N}}$.
We remark that, as indicated in the above isomorphisms, $\Omega_{\T_{\G_{N}}}^{\theta}$ can be identified with $\Omega_{\T_{\U_{N}}}$ by noting that $\U_{N}$ is equal to $\G_{N}^{\theta}$ 
(see, for example, \cite[Section 1.1]{MR1687096} for a general explanation of this identification).
From this description, we get the following lemma:
\begin{lem}\label{lem:center}
Let $\gamma \in \U_{N}^{\srs}(F)$ and $\delta \in \G_{N}^{\strs}(F)$ such that $\gamma$ is a norm of $\delta$.
Then, for $z\in E^{\times}$, we have $z/c(z)\cdot\gamma\in \U_{N}^{\srs}(F)$ and $z\delta\in \G_{N}^{\strs}(F)$.
Moreover, $z/c(z)\cdot\gamma$ is a norm of $z\delta$. 
\end{lem}

\begin{proof}
Recall that $\gamma\in \U_{N}(F)$ is said to be strongly regular semisimple if its centralizer in $\U_{N}$ is a maximal torus of $\U_{N}$, and that $\delta\in \G_{N}(F)$ is said to be strongly $\theta$-regular $\theta$-semisimple if its $\theta$-centralizer in $\G_{N}$ is abelian.
Then the first assertion is trivial because a scalar-multiple does not affect ($\theta$-)centralizers.
The second assertion follows from the above explicit description of the map $\mathcal{A}_{\U_{N}/\G_{N}}$.
\end{proof}

In the same manner as Proposition 4.7 in \cite{Oi:2016}, we get the following proposition:
\begin{prop}\label{prop:affgennorm}
Let $h \in I_{\U_{N}}^{+} \subset \U_{N}(F)$ be an affine generic element.
Then $h$ is strongly regular semisimple elliptic, and there exists $g \in \G_{N}(F)$ satisfying the following conditions:
\begin{itemize}
 \item $g$ is strongly $\theta$-regular $\theta$-semisimple $\theta$-elliptic, and
 \item $h=\mathcal{N}(g)$, in particular $h$ is a norm of $g$.
\end{itemize}
\end{prop}

\subsection{Transfer factors}
In this subsection, we show that Kottwitz--Shelstad's transfer for the standard base change embedding $\xi_{+1}$ is trivial.
We fix the following $\theta$-stable Whittaker datum $(\mathbf{B}_{\G_{N}}, \lambda_{\G_{N}})$ of $\G_{N}$: 
\begin{itemize}
 \item $\mathbf{B}_{\G_{N}}$ is the subgroup of upper triangular matrices in $\G_{N}$, and 
 \item $\lambda_{\G_{N}}$ is the character of the unipotent radical $U_{\G_{N}}$ of $B_{\G_{N}}$ defined by 
 \[
 \lambda_{\G_{N}}(x):=\psi\circ\Tr_{E/F}(x_{12}+\cdots+x_{N-1, N}) \quad\text{for}\quad x=(x_{ij}) \in U_{\G_{N}}.
 \]
\end{itemize}
Then we have the normalized transfer factor $\Delta_{\U_{N},\G_{N}}$ for $\G_{N}$ and $\U_{N}$ with respect to $(\mathbf{B}_{\G_{N}}, \lambda_{\G_{N}})$ (see Section 5.3 in \cite{MR1687096}).
This is a function
\[
\Delta_{\U_{N},\G_{N}} \colon \U_{N}^{\srs}(F) \times \G_{N}^{\strs}(F) \ra \C,
\]
which has the following properties.
\begin{itemize}
 \item The value $\Delta_{\U_{N},\G_{N}}(h, g)$ is nonzero only if $h$ is a norm of $g$.
 \item If $h_1, h_2 \in \U_{N}^{\srs}(F)$ are stably conjugate, then $\Delta_{\U_{N},\G_{N}}(h_1, g)=\Delta_{\U_{N},\G_{N}}(h_2, g)$.
 \item If $g_1, g_2 \in \G_{N}^{\strs}(F)$ are $\theta$-conjugate, then $\Delta_{\U_{N},\G_{N}}(h, g_1)=\Delta_{\U_{N},\G_{N}}(h, g_2)$.
\end{itemize}

Recall that the transfer factor $\Delta_{\U_{N},\G_{N}}$ is defined as the product of $\Delta_\mathrm{I}$, $\Delta_\mathrm{II}$, $\Delta_\mathrm{III}$, $\Delta_\mathrm{IV}$, and the ratio of the root numbers of the root data for $\G_{N}$ and $\U_{N}$.

However, by using Waldspurger's formula for the transfer factors for classical groups (\cite[1.10 Proposition]{MR2672539}), 
we know that the product of the three factors $\Delta_{\mathrm{I}}$, $\Delta_{\mathrm{II}}$, and $\Delta_{\mathrm{III}}$ is trivial.
Indeed, in the notation of \cite{MR2672539}, only $H^{-}$-part of the endoscopic group $H=H^{+}\times H^{-}$ of $\G_{N}$ contributes to the product $\Delta_{\mathrm{I}}\cdot\Delta_{\mathrm{II}}\cdot\Delta_{\mathrm{III}}$ nontrivially.
However, in our setting, the $H^{-}$-part is trivial and we have $H=H^{+}=\U_{N}$ (see \cite[1.8 and 1.10]{MR2672539} for details).

Moreover, from the same argument as in Lemma 4.10 in \cite{Oi:2016}, we can check the triviality of $\Delta_{\mathrm{IV}}$ directly.

Finally, we can also check the triviality of the ratio of the root numbers.
Namely, we have
\[
\varepsilon_{\U_{N},\G_{N}}:=\frac{\varepsilon\bigl(\frac{1}{2},X^{\ast}(\mathbf{T}_{\G_{N}})^{\theta}\otimes_{\Z}\C,\psi\bigr)}{\varepsilon\bigl(\frac{1}{2},X^{\ast}(\mathbf{T}_{\U_{N}})\otimes_{\Z}\C,\psi\bigr)}=1,
\]
since we have $X^{\ast}(\T_{\G_{N}})^{\theta}\cong X^{\ast}(\T_{\G_{N},\theta})$ and $\T_{\G_{N},\theta}\cong\T_{\U_{N}}(\cong\mathrm{Res}_{E/F}\Gm^{N}$).

In summary, we get the following:
\begin{prop}\label{prop:Delta}
Let $h \in \U_{N}^{\srs}(F)$ and $g \in \G_{N}^{\strs}(F)$.
If $h$ is a norm of $g$, then the transfer factor $\Delta_{\U_{N},\G_{N}}(h, g)$ is equal to $1$.
\end{prop}

\subsection{Conjugate self-dual $L$-parameters}
We denote by $\Phi(\GL_{N,E})$ and $\Phi(\G_{N})$ the sets of equivalence classes of $L$-parameters of $\GL_{N}$ over $E$ and $\G_{N}$, respectively.
We fix $w_{c} \in W_{F}\setminus W_{E}$ and 
define a map $\phi\mapsto\phi'$ from $\Phi(\GL_{N,E})$ to $\Phi(\G_{N})$ as:
\begin{align*}
\phi' \colon W_{F}\times\SL_{2}(\C) &\ra \bigl(\GL_{N}(\C)\times\GL_{N}(\C) \bigr)\rtimes W_{F}={}^{L}\G_{N}\\
\sigma &\mapsto \left(\phi(\sigma), \phi(w_{c}^{-1}\sigma w_{c})\right)\rtimes\sigma \quad\text{for}\quad \sigma \in  W_{E},\\
w_{c} &\mapsto \left(\phi(w_{c}^{2}), I_{N}\right)\rtimes w_{c}, \\
g &\mapsto \bigl(\phi(g),\phi(g)\bigr)\rtimes1\quad\text{for}\quad g \in\SL_{2}(\C).
\end{align*}
Here we regard $\phi\in\Phi(\GL_{N,E})$ as an $N$-dimensional representation of $W_{E}\times\SL_{2}(\C)$.
Then this map is bijective and independent of the choice of $w_{c}$ (see Section 4.7 in \cite{MR1081540} or Section 2.2 in \cite{MR3338302}).
We often identify $\Phi(\GL_{N,E})$ with $\Phi(\G_{N})$ via this map.

\begin{defn}
\begin{enumerate}
\item
We say that an $L$-parameter $\phi \in \Phi(\G_{N})$ is conjugate self-dual if $\phi^{\vee}$ is equivalent to $\phi^{c}$ as $N$-dimensional representations of $W_{E}\times\SL_{2}(\C)$, where 
\begin{align*}
\phi^{\vee}(\sigma)&:={}^{t}\!\phi(\sigma)^{-1}, \text{ and} \\
\phi^{c}(\sigma)&:=\phi(w_{c}^{-1}\sigma w_{c}).
\end{align*}
\item
We say that a conjugate self-dual $L$-parameter $\phi \in \Phi(\G_{N})$ has parity $\eta\in\{\pm1\}$ if there exists $A \in \GL_{N}(\C)$ such that
\[
{}^{t}\!\phi^{c}(\sigma)\cdot A\cdot\phi(\sigma)=A \quad\text{and}\quad
{}^{t}\!A=\eta \cdot A \cdot \phi(w_{c}^{2}).
\]
\end{enumerate}
\end{defn}
Note that, if a conjugate self-dual $L$-parameter $\phi\in\Phi(\G_{N})$ is irreducible, then it has its parity as a consequence of Schur's lemma.
%

Next we consider the group $\U_{N}$.
Let $\Phi(\U_{N})$ be the set of equivalence classes of $L$-parameters of $\U_{N}$ and $\xi=\xi_\kappa$ the $L$-embedding defined in Section \ref{sec:endo}.1 for $\kappa=\pm1$.
Then we get a map 
\[
\xi_{\kappa, \ast}\colon\Phi(\U_{N})\ra\Phi(\G_{N});\quad \phi\mapsto\xi_{\kappa}\circ\phi.
\]
The image of this map is described as follows:

\begin{lem}\label{lem:csdLpar}
\begin{enumerate}
 \item The map $\xi_{\kappa,\ast}$ is injective and its image is the set of conjugate self-dual $L$-parameters of $\G_{N}$ with parity $(-1)^{N-1}\kappa$.
 \item Let $\phi\in\Phi(\U_{N})$ and $\phi_{\kappa}:=\xi_{\kappa, \ast}(\phi)$.
 We denote by $\pi_{\phi_{\kappa}}$ the irreducible smooth representation of $\GL_{N}(E)$ which corresponds to $\phi_{\kappa}$ via the local Langlands correspondence for $\GL_{N}(E)$.
 Then we have $\pi_{\phi_{-1}}=\pi_{\phi_{+1}}\otimes(\chi\circ\det)$.
\end{enumerate}
\end{lem}

\begin{proof}
The first assertion is Lemma 2.2.1 in \cite{MR3338302}.
The second assertion follows easily from the definition of the above bijection $\Phi(\G_{N})\cong\Phi(\GL_{N,E})$ and a property of the local Langlands correspondence for $\GL_{N}$.
\end{proof}

\begin{lem}\label{lem:csdssc}
Let $(\omega,a,\zeta)\in\SSC^{\theta}(\G_{N})$.
Then we have 
\[
\pi^{\G_{N}}_{\omega,a,\zeta}\cong\pi^{\G_{N}}_{\omega,a,-\zeta}\otimes (\chi\circ\det),
\]
where $\chi=(-1)^{\mathrm{val}}$ is the character of $E^{\times}$.
\end{lem}

\begin{proof}
By the Frobenius reciprocity, we have
\begin{align*}
&\Hom_{\G_{N}(F)}\bigl(\pi^{\G_{N}}_{\omega,a,\zeta}, \pi^{\G_{N}}_{\omega,a,-\zeta}\otimes (\chi\circ\det)\bigr)\\
&\cong \Hom_{Z_{\G_{N}}I_{\G_{N}}^{+}\lan\varphi_{a^{-1}}\ran}\bigl(\chi^{\G_{N}}_{\omega,a,\zeta}, \pi^{\G_{N}}_{\omega,a,-\zeta}\otimes (\chi\circ\det)\bigr)\\
&\cong \Hom_{Z_{\G_{N}}I_{\G_{N}}^{+}\lan\varphi_{a^{-1}}\ran}\bigl(\chi^{\G_{N}}_{\omega,a,\zeta}\otimes (\chi^{-1}\circ\det), \pi^{\G_{N}}_{\omega,a,-\zeta}\bigr)\\
&= \Hom_{Z_{\G_{N}}I_{\G_{N}}^{+}\lan\varphi_{a^{-1}}\ran}\bigl(\chi^{\G_{N}}_{\omega,a,-\zeta}, \pi^{\G_{N}}_{\omega,a,-\zeta}\bigr)\neq0.
\end{align*}
Since $\pi^{\G_{N}}_{\omega,a,\zeta}$ and $\pi^{\G_{N}}_{\omega,a,-\zeta}\otimes (\chi\circ\det)$ are irreducible, they are equivalent.
\end{proof}

\subsection{Endoscopic character relation}
We first recall the endoscopic classification of representations of unitary groups in \cite{MR3338302}.

For $\phi \in \Phi(\U_{N})$, we set
\begin{align*}
 S_{\phi} &:= \Cent_{\widehat{\U_{N}}}\bigl(\mathrm{Im}(\phi)\bigr) \text{, and }\\
 \mathcal{S}_{\phi} &:= S_{\phi} / S_{\phi}^0 Z(\widehat{\U_{N}})^{W_{F}}.
\end{align*}

\begin{thm}[{\cite[Theorems 2.5.1 and 3.2.1]{MR3338302}}]\label{thm:Mok}
For every tempered $L$-parameter $\phi \in \Phi(\U_{N})$, there is a finite set $\Pi^{\U_{N}}_{\phi}$ consisting of irreducible tempered representations of $\U_{N}(F)$, and the following properties hold.
\begin{itemize}
\item 
There is a bijection from $\Pi^{\U_{N}}_{\phi}$ to the group $\widehat{\mathcal{S}_{\phi}}$ of characters of $\mathcal{S}_{\phi}$ $($depending on a choice of Whittaker data for $\U_{N}$$)$.
\item
The sum of the characters of the representations belonging to $\Pi_{\phi}^{\U_{N}}$ is stable.
\item
Let $\phi_{\kappa}:=\xi_{\kappa,\ast}(\phi)$ and $\pi^{\G_{N}}_{\kappa}$ be the corresponding conjugate self-dual supercuspidal representation of $\G_{N}(F)$.
Then, for every $g \in \G_{N}^{\strs}(F)$, we have the following equality: 
\[
\Theta^{\G_{N}}_{\phi,\kappa,\theta}(g) = \sum_{h\mapsto g/\sim} \frac{D_{\U_{N}}(h)^2}{D_{\G_{N},\theta}(g)^2} \Delta_{\U_{N},\G_{N}}(h,g) \sum_{\pi\in\Pi_{\phi}^{\U_{N}}}\Theta_{\pi}(h).
\]
Here,
\begin{itemize}
\item the sum is over stable conjugacy classes of norms $h \in \U_{N}^{\srs}(F)$ of $g$,
\item $\Theta^{\G_{N}}_{\phi,\kappa,\theta}$ is the $\theta$-twisted character of $\pi^{\G_{N}}_{\kappa}$ normalized via a fixed $\theta$-stable Whittaker data $($see Remark $\ref{rem:normalization}$$)$, 
\item $D_{\U_{N}}(h)$ $($resp.\ $D_{\G_{N},\theta}(g)$$)$ is the Weyl discriminant for $h$ $($resp.\ the $\theta$-twisted Weyl discriminant for $g$$)$, and 
\item $\Delta_{\U_{N},\G_{N}}$ is the Kottwitz--Shelstad transfer factor $($with respect to $\xi_{\kappa}$$)$ normalized via a fixed $\theta$-stable Whittaker data.
\end{itemize}
\end{itemize}
\end{thm}

Here we note, in {\cite[Theorem 3.2.1]{MR3338302}, the above relation of characters are stated in terms of the distribution characters.
However, we can rewrite it as an equality of characters (i.e., functions on the sets of strongly regular semisimple elements of $\U_{N(F)}$ and $\G_{N}(F)$) by using Weyl's integration formula (e.g., see Section 5 in \cite{MR3067291}).

Let $\pi$ be a conjugate self-dual irreducible supercuspidal representation of $\G_{N}(F)$ such that the parity of the corresponding $L$-parameter $\phi_{\pi}$ is given by $(-1)^{N-1}\kappa$.
From Lemma \ref{lem:csdLpar} (1), $\phi'_{\pi}$ factors through the $L$-embedding $\xi_{\kappa}$.
We write $\phi$ for the $L$-parameter of $\U_{N}$ such that $\phi'_{\pi}=\xi_{\kappa}\circ\phi$.
Then, by Theorem \ref{thm:Mok}, we get a finite set $\Pi^{\U_{N}}_{\phi}$ consisting of irreducible representations of $\U_{N}(F)$.

Since $\pi$ is supercuspidal, its $L$-parameter $\phi_{\pi}$ is irreducible as a representation of $W_{E}\times\SL_{2}(\C)$.
Therefore $\Cent_{\widehat{\G_{N}}}(\mathrm{Im}(\phi'_{\pi})) \subset \GL_{N}(\C)\times\GL_{N}(\C)$ consists of pairs of scalar matrices.
Thus $\Cent_{\widehat{\U_{N}}}(\mathrm{Im}(\phi))$ consists of scalar matrices.
Hence the group $\mathcal{S}_{\phi}$ is trivial and $\Pi^{\U_{N}}_{\phi}$ is a singleton by Theorem \ref{thm:Mok}.

Moreover, since the $L$-packet $\Pi^{\U_{N}}_{\phi}$ contains a supercuspidal representation by \cite[8.4.4 $\mathrm{Th\acute{e}or\grave{e}me}$]{MR2366373}, the unique representation in $\Pi^{\U_{N}}_{\phi}$ is supercuspidal.
We denote it by $\pi_{\U_{N}}$ and say that $\pi_{\U_{N}}$ is $\mathit{associated}$ $\mathit{to}$ $\pi$.
We remark that the character $\Theta_{\pi_{\U_{N}}}$ of $\pi_{\U_{N}}$ is stable by Theorem \ref{thm:Mok}.
\[
\xymatrix{
\pi & \ar@{<~>}[r]^-{\text{LLC for $\GL_{N}$}} && W_F\times\SL_{2}(\C) \ar[rd]_-{\phi} \ar[r]^-{\phi'_{\pi}} & {}^{L}\G_{N}\\
{\Pi^{\U_{N}}_{\phi}=\{\pi_{\U_{N}}\}} & \ar@{<~>}[r]^-{\text{Theorem \ref{thm:Mok}}} &&& {}^{L}\U_{N} \ar@{^{(}->}_-{\xi_{\kappa}}[u] \\
}
\]

We can simplify this equality by noting the bijectivity of $\mathcal{A}_{\U_{N}/\G_{N}}$ and the triviality of the transfer factor $\Delta_{\U_{N},\G_{N}}$:
\begin{cor}\label{cor:ECR}
Let $\pi$ be a conjugate self-dual irreducible supercuspidal representation of $\G_{N}(F)$.
We suppose that the parity of the $L$-parameter $\phi_{\pi}$ of $\pi$ is $(-1)^{N-1}$ $($i.e., $\phi'_{\pi}$ factors through the standard base change $L$-embedding $\xi_{+1}$$)$.
Let $\pi_{\U_{N}}$ be the irreducible supercuspidal representation of $\U_{N}(F)$ associated to $\pi$ in the above manner.
Let $g \in \G_{N}^{\strs}(F)$ and $h \in \U_{N}^{\srs}(F)$ such that $h$ is a norm of $g$.
Then we have 
\[
\Theta_{\pi, \theta}(g)=\Theta_{\pi_{\U_{N}}}(h).
\]
\end{cor}

\begin{proof}
Since the map $\mathcal{A}_{\U_{N}/\G_{N}}$ is bijective, every element of $\G_{N}^{\strs}(F)$ has at most one norm in $\U_{N}^{\srs}(F)$ up to stable conjugacy.
On the other hand, for $(h,g)$, the transfer factor is trivial by Proposition \ref{prop:Delta}.
Moreover, also $D_{\U_{N}}(h)/D_{\G_{N},\theta}(g)$ is trivial since this equals $\Delta_{\mathrm{IV}}(h,g)^{-1}$.
Therefore, by combining these observations with Theorem \ref{thm:Mok}, we get the assertion.
\end{proof}

\section{Endoscopic lifting of simple supercuspidal representations}\label{sec:main}
\subsection{Norms of $1+\varphi_{\tilde{u}}$ and $\varphi_{\tilde{u}}(1+\varphi_{\tilde{u}})$}

We first introduce two technical lemmas about the $\theta$-semisimplicity and $\theta$-regularity of elements of $\G_{N}(\ol{F})$.
These lemmas can be shown by the same arguments as in \cite[Lemma 4.3]{Oi:2016} and \cite[Lemma 4.11]{Oi:2016}:

\begin{lem}\label{lem:norm0}
Let $x$ be a $\theta$-semisimple element in $\G_{N}(\ol{F})$ and $y$ a semisimple element in $\U_{N}(\ol{F})$ which corresponds to $x$ via $\mathcal{A}_{\U_{N}/\G_{N}}$.
If $\mathcal{N}(x)=x\theta(x) \in \G_{N}(\ol{F})$ is regular, then $x$ is strongly $\theta$-regular and $y$ is strongly regular.
\end{lem}

\begin{lem}\label{lem:tss}
Let $g \in \G_{N}(\ol{F})$ be a semisimple element such that $\mathcal{N}(g)=g\theta(g)$ is regular semisimple.
If $g\theta(g)=\theta(g)g$, then $g$ is $\theta$-semisimple.
\end{lem}

Now let $\tilde{u}$ be an element of
\[
\begin{cases}
k^{\times} & \text{if } N=2n\\
(\tilde{k}^{0})^{\times} & \text{if } N=2n+1,
\end{cases}
\]
and we consider the following matrices:
\[
1+\varphi_{\tilde{u}} = 
\begin{pmatrix}
 1&1&&\\
 &1&\ddots&\\
 &&\ddots&1\\
 \varpi \tilde{u}&&&1
\end{pmatrix} 
\in I_{\G_{N}}^{+} \subset \G_{N}(F) \text{, and }
\]
\[
\mathcal{N}(1+\varphi_{\tilde{u}})= 
(1-\varpi \tilde{u})^{-1}
\begin{pmatrix}
 1+\varpi \tilde{u}&2&\hdots&2\\
 2\varpi \tilde{u}&\ddots&\ddots&\vdots\\
 \vdots&\ddots&\ddots&2\\
 2\varpi \tilde{u}&\hdots&2\varpi \tilde{u}&1+\varpi \tilde{u}
\end{pmatrix}\\
\in I_{\U_{N}}^{+}\subset \U_{N}(F).
\]
Note that these are affine generic elements.

\begin{prop}\label{prop:N}
The element $1+\varphi_{\tilde{u}} \in \G_{N}(F)$ is strongly $\theta$-regular $\theta$-semisimple and $\mathcal{N}(1+\varphi_{\tilde{u}}) \in \U_{N}(F)$ is strongly regular semisimple.
Moreover, $\mathcal{N}(1+\varphi_{\tilde{u}})$ is a norm of $1+\varphi_{\tilde{u}}$.
\end{prop}

\begin{proof}
We first show that $1+\varphi_{\tilde{u}}$ is $\theta$-semisimple.
As the characteristic polynomial of an affine generic element is Eisenstein over $E$ (see, e.g., Lemma 7.5 in \cite{Oi:2018}), $1+\varphi_{\tilde{u}}$ has $N$ distinct eigenvalues.
In particular, $1+\varphi_{\tilde{u}}$ is semisimple.
Since 
\begin{align*}
\theta(1+\varphi_{\tilde{u}})
&=J_{N}(1+{}^tc(\varphi_{\tilde{u}}))^{-1}J_{N}^{-1}\\
&=J_{N}(1-{}^t\varphi_{(-1)^{N}\tilde{u}}+{}^t\varphi_{(-1)^{N}\tilde{u}}^2-\cdots)J_{N}^{-1}\\
&=1+\varphi_{\tilde{u}}+\varphi_{\tilde{u}}^2+\cdots,
\end{align*}
$1+\varphi_{\tilde{u}}$ commutes with $\theta(1+\varphi_{\tilde{u}})$.
Moreover, since the residual characteristic is not equal to $2$, $\mathcal{N}(1+\varphi_{\tilde{u}})$ is an affine generic element of $\G_{N}(F)$.
Hence $\mathcal{N}(1+\varphi_{\tilde{u}})$ is regular semisimple from the same reason as $1+\varphi_{\tilde{u}}$.
Therefore $1+\varphi_{\tilde{u}}$ is $\theta$-semisimple by Lemma \ref{lem:tss}.

By the definition of $\mathcal{A}_{\U_{N}/\G_{N}}$, $\mathcal{N}(1+\varphi_{\tilde{u}})$ corresponds to $1+\varphi_{\tilde{u}}$ via $\mathcal{A}_{\U_{N}/\G_{N}}$.
Then $1+\varphi_{\tilde{u}}$ is strongly $\theta$-regular and $\mathcal{N}(1+\varphi_{\tilde{u}})$ is strongly regular by Lemma \ref{lem:norm0}.
\end{proof}

Next we consider the following elements:
\[
\varphi_{\tilde{u}}(1+\varphi_{\tilde{u}}) \in \G_{N}(F) \text{, and }
\]
\[
\mathcal{N}\bigl(\varphi_{\tilde{u}}(1+\varphi_{\tilde{u}})\bigr)=\bigl(\varphi_{\tilde{u}}(1+\varphi_{\tilde{u}})\bigr)\cdot\theta\bigl(\varphi_{\tilde{u}}(1+\varphi_{\tilde{u}})\bigr)=-\mathcal{N}(1+\varphi_{\tilde{u}}) \in \U_{N}(F).
\]

\begin{prop}\label{prop:N'}
The element $\varphi_{\tilde{u}}(1+\varphi_{\tilde{u}}) \in \G_{N}(F)$ is strongly $\theta$-regular $\theta$-semisimple and $-\mathcal{N}(1+\varphi_{\tilde{u}}) \in \U_{N}(F)$ is strongly regular semisimple.
Moreover, $-\mathcal{N}(1+\varphi_{\tilde{u}})$ is a norm of $\varphi_{\tilde{u}}(1+\varphi_{\tilde{u}})$.
\end{prop}

\begin{proof}
By the same argument in the proof of Proposition \ref{prop:N}, we only have to show that $\varphi_{\tilde{u}}(1+\varphi_{\tilde{u}})$ is $\theta$-semisimple.
Since $1+\varphi_{\tilde{u}}$ is semisimple by the proof of the previous proposition,  it can be diagonalized in $\G_{N}(\ol{F})$.
Hence so is $\varphi_{\tilde{u}}$.
As $\varphi_{\tilde{u}}$ commutes with $1+\varphi_{\tilde{u}}$, also $\varphi_{\tilde{u}}(1+\varphi_{\tilde{u}})$ can be diagonalized.
Thus $\varphi_{\tilde{u}}(1+\varphi_{\tilde{u}})$ is a semisimple element.
Since 
\begin{align*}
\theta\bigl(\varphi_{\tilde{u}}(1+\varphi_{\tilde{u}})\bigr)
&=\theta(\varphi_{\tilde{u}})\theta(1+\varphi_{\tilde{u}})\\
&=-\varphi_{\tilde{u}}^{-1}(1+\varphi_{\tilde{u}}+\varphi_{\tilde{u}}^2+\cdots),
\end{align*}
$\varphi_{\tilde{u}}(1+\varphi_{\tilde{u}})$ commutes with $\theta(\varphi_{\tilde{u}}(1+\varphi_{\tilde{u}}))$.
Finally, as
\[
\mathcal{N}\bigl(\varphi_{\tilde{u}}(1+\varphi_{\tilde{u}})\bigr)= -\mathcal{N}(1+\varphi_{\tilde{u}})
\]
is regular semisimple by the proof of Proposition \ref{prop:N}, $\varphi_{\tilde{u}}(1+\varphi_{\tilde{u}})$ is $\theta$-semisimple by Lemma \ref{lem:tss}.
\end{proof}

%

\subsection{Parity of simple supercuspidal representations}
Let $(\omega,a,\zeta)\in \SSC^{\theta}(\G_{N})$.
From Lemmas \ref{lem:csdLpar} and \ref{lem:csdssc}, we know that 
the parity of the $L$-parameter $\pi^{\G_{N}}_{\omega,a,1}$ is different from that of $\pi^{\G_{N}}_{\omega,a,-1}$.
In this subsection, we determine the parity of these parameters.
We first consider the case where 
\[
a=
\begin{cases}
1 &\text{if } N=2n,\\
\epsilon &\text{if } N=2n+1.
\end{cases}
\]

Let $\zeta_{\omega} \in \{\pm1\}$ such that the parity of the $L$-parameter of $\pi^{\G_{N}}_{\omega,a,\zeta_{\omega}}$ is $(-1)^{N-1}$.
Then the $L$-parameter of $\pi^{\G_{N}}_{\omega,a,\zeta_{\omega}}$ factors through the standard base change embedding $\xi_{+1}$ by Lemma \ref{lem:csdLpar}.
Let $\pi_{\U_{N}}$ be the supercuspidal representation of $\U_{N}(F)$ which is associated to $\pi^{\G_{N}}_{\omega,a,\zeta_{\omega}}$.

Let $\tilde{u}$ be an element of
\[
\begin{cases}
k^{\times} & \text{if } N=2n,\\
(\tilde{k}^{0})^{\times} & \text{if } N=2n+1.
\end{cases}
\]

\begin{prop}\label{prop:CR}
For $z\in \mcO_{E}^{\times}$, we have 
\begin{align*}
\Theta_{\pi_{\U_{N}}}\bigl(z/c(z)\cdot \mathcal{N}(1+\varphi_{\tilde{u}})\bigr)
&=\Theta^{\G_{N}}_{\omega,a,\zeta_{\omega},\theta}\bigl(z(1+\varphi_{\tilde{u}})\bigr)\\
&=
\begin{cases}
-\omega(\ol{z})\cdot\Kl_{2^{2n}\tilde{u}}^{0;n}(\psi) & \text{if } N=2n,\\
\omega(\ol{z})\cdot\Kl_{2^{2n+1}\epsilon\tilde{u}}^{1;n}(\psi) & \text{if } N=2n+1.
\end{cases}
\end{align*}
\end{prop}

\begin{proof}
By Proposition \ref{prop:N} and Lemma \ref{lem:center}, $z/c(z)\cdot \mathcal{N}(1+\varphi_{\tilde{u}}) \in \U_{N}^{\srs}(F)$ is a norm of $z(1+\varphi_{\tilde{u}}) \in \G_{N}^{\strs}(F)$.
Therefore, by Corollary \ref{cor:ECR}, we have 
\[
\Theta^{\G_{N}}_{\omega,a,\zeta_{\omega},\theta}\bigl(z(1+\varphi_{\tilde{u}})\bigr)
=\Theta_{\pi_{\U_{N}}}\bigl(z/c(z)\cdot \mathcal{N}(1+\varphi_{\tilde{u}})\bigr).
\]
On the other hand, since we have $\omega^{\G_{N}}_{\omega,a,\zeta_{\omega}}(z)=\omega(\ol{z})$, where $\omega^{\G_{N}}_{\omega,a,\zeta_{\omega}}$ is the central character of $\pi^{\G_{N}}_{\omega,a,\zeta_{\omega}}$, we get
\[
\Theta^{\G_{N}}_{\omega,a,\zeta_{\omega},\theta}\bigl(z(1+\varphi_{\tilde{u}})\bigr)
=\omega(\ol{z})\cdot\Theta^{\G_{N}}_{\omega,a,\zeta_{\omega},\theta}(1+\varphi_{\tilde{u}}).
\]

Here we note that, since the Whittaker normalization coincides with the normalization used in Section \ref{sec:char} (see Remark \ref{rem:normalization}), we can apply the computations in Section \ref{sec:char} to the right-hand sides of the above equalities.
Since $\mathcal{N}(1+\varphi_{\tilde{u}})$ is an affine generic element of $\G_{N}(F)$, $1+\varphi_{\tilde{u}}$ satisfies the assumptions of Propositions \ref{prop:charTG} and \ref{prop:charTG_{o}}.
As the simple affine components of $1+\varphi_{\tilde{u}}$ are given by $(1,\ldots,1,\tilde{u})$, we have
\[
\Theta^{\G_{N}}_{\omega,a,\zeta_{\omega},\theta}(1+\varphi_{\tilde{u}})
=
\begin{cases}
-\Kl_{2^{2n}\tilde{u}}^{0;n}(\psi) & \text{if } N=2n,\\
 \Kl_{2^{2n+1}\epsilon\tilde{u}}^{1;n}(\psi) & \text{if } N=2n+1.
\end{cases}
\]
\end{proof}

\begin{prop}\label{prop:CR'}
We have
\begin{align*}
\Theta_{\pi_{\U_{N}}}\bigl(-\mathcal{N}(1+\varphi_{\tilde{u}})\bigr)
&=\Theta^{\G_{N}}_{\omega,a,\zeta_{\omega},\theta}\bigl(\varphi_{\tilde{u}}(1+\varphi_{\tilde{u}})\bigr)\\
&=\zeta_{\omega}\cdot
\begin{cases}
\Kl_{2^{2n}\tilde{u}}^{0;n}(\psi) & \text{if } N=2n,\\
 \Kl_{2^{2n+1}\epsilon\tilde{u}}^{1;n}(\psi) & \text{if } N=2n+1.
\end{cases}
\end{align*}
\end{prop}

\begin{proof}
By Proposition \ref{prop:N'}, $-\mathcal{N}(1+\varphi_{\tilde{u}}) \in \U_{N}^{\srs}(F)$ is a norm of $\varphi_{\tilde{u}}(1+\varphi_{\tilde{u}}) \in \G_{N}^{\strs}(F)$.
Therefore, by Corollary \ref{cor:ECR}, we have 
\[
\Theta^{\G_{N}}_{\omega,a,\zeta_{\omega},\theta}\bigl(\varphi_{\tilde{u}}(1+\varphi_{\tilde{u}})\bigr)
=\Theta_{\pi_{\U_{N}}}\bigl(-\mathcal{N}(1+\varphi_{\tilde{u}})\bigr).
\]
Since $1+\varphi_{\tilde{u}}$ satisfies the assumptions of Propositions \ref{prop:charTG'} and \ref{prop:charTG'_{o}} and the simple affine components of $1+\varphi_{\tilde{u}}$ are given by $(1,\ldots,1,\tilde{u})$, we have
\[
\Theta^{\G_{N}}_{\omega,a,\zeta_{\omega},\theta}\bigl(\varphi_{\tilde{u}}(1+\varphi_{\tilde{u}})\bigr)
=\zeta_{\omega}\cdot
\begin{cases}
\Kl_{2^{2n}\tilde{u}}^{0;n}(\psi) & \text{if } N=2n,\\
 \Kl_{2^{2n+1}\epsilon\tilde{u}}^{1;n}(\psi) & \text{if } N=2n+1.
\end{cases}
\]
\end{proof}

\begin{cor}\label{cor:parity}
We have $\zeta_{\omega}=(-1)^{N-1}\omega(\epsilon)$.
\end{cor}

\begin{proof}
By applying Propositions \ref{prop:CR} and \ref{prop:CR'} to $z=\epsilon$, for any $\tilde{u}$, we have 
\[
\begin{cases}
-\omega(\epsilon)\cdot\Kl_{2^{2n}\tilde{u}}^{0;n}(\psi)=\zeta_{\omega}\cdot\Kl_{2^{2n}\tilde{u}}^{0;n}(\psi) & \text{if } N=2n,\\
\omega(\epsilon)\cdot\Kl_{2^{2n+1}\epsilon\tilde{u}}^{1;n}(\psi)=\zeta_{\omega}\cdot\Kl_{2^{2n+1}\epsilon\tilde{u}}^{1;n}(\psi) & \text{if } N=2n+1.
\end{cases}
\]
Since there exists a $\tilde{u}$ such that the Kloosterman sum does not vanish (see, e.g., \cite[Proposition A.2]{Oi:2018}), we get $\zeta_{\omega}=(-1)^{N-1}\omega(\epsilon)$.
\end{proof}

Finally, by replacing the fixed uniformizer $\varpi$ with its scalar multiple by $k^{\times}$, we can conclude that the parity of the $L$-parameter of $\pi^{\G_{N}}_{\omega,a,\zeta_{\omega}}$ is $(-1)^{N-1}$ for any $a$.

\subsection{Existence of $I_{\U_{N}}^{++}$-fixed vector}
We again put
\[
a=
\begin{cases}
1 &\text{if } N=2n,\\
\epsilon &\text{if } N=2n+1,
\end{cases}
\]
and consider the simple supercuspidal representation $\pi^{\G_{N}}_{\omega,a,\zeta_{\omega}}$.
Let $\pi_{\U_{N}}$ be the irreducible supercuspidal representation of $\U_{N}(F)$ associated to $\pi^{\G_{N}}_{\omega,a,\zeta_{\omega}}$.

\begin{lem}\label{lem:charlift}
Let $h \in \U_{N}(F)$ be an affine generic element with its simple affine components $(h_1, \ldots, h_n, h_{0})$.
Then $\Theta_{\pi_{\U_{N}}}(h)$ is equal to either $0$ or 
\[
\begin{cases}
-\Kl^{0;n}_{\Nr(h_{1})\cdots\Nr(h_{n-1})h_{n}h_{0}}(\psi) & \text{if } N=2n,\\
\Kl^{1;n}_{\Nr(h_{1})\cdots\Nr(h_{n})h_{0}\epsilon}(\psi) & \text{if } N=2n+1.
\end{cases}
\]
\end{lem}

\begin{proof}
For such $h$, we take $g \in \G_{N}(F)$ satisfying the conditions in Proposition \ref{prop:affgennorm} (in particular, we have $\mathcal{N}(g)=h$).
Since $h \in \U_{N}^{\srs}(F)$, $g \in \G_{N}^{\strs}(F)$, and $h$ is a norm of $g$, 
we have 
\[
\Theta^{\G_{N}}_{\omega,a,\zeta_{\omega},\theta}(g)
=\Theta_{\pi_{\U_{N}}}(h)
\]
by Corollary \ref{cor:ECR} (note that, from the results of Section \ref{sec:main}.2, $\pi^{\G_{N}}_{\omega,a,\zeta_{\omega}}$ is the endoscopic lift of $\pi_{\U_{N}}$ via the standard base change embedding $\xi_{+1}$).
We compute the left-hand side of this equality.

If there is no $x \in \G_{N}(F)$ such that $xg\theta(x)^{-1} \in Z_{\G_{N}}I_{\G_{N}}^{+}\lan\varphi_{a^{-1}}\ran$, 
then, by the twisted character formula (Theorem \ref{thm:TCF}), we have 
\[
\Theta^{\G_{N}}_{\omega,a,\zeta_{\omega},\theta}(g)=0.
\]

Let us consider the case where there exists $x \in \G_{N}(F)$ such that 
\[
xg\theta(x)^{-1} \in Z_{\G_{N}}I_{\G_{N}}^{+}\lan\varphi_{a^{-1}}\ran.
\]
By Lemma \ref{lem:sumTG}, we may assume $x \in T_{\G_{N}}(q)$.
Since we have $\theta(\varphi_{a^{-1}})=-\varphi_{a^{-1}}^{-1}$ and $\varphi_{a^{-1}}$ normalizes $I_{\G_{N}}^{+}$, 
we have 
\[
\varphi_{a^{-1}}^k xg\theta(\varphi_{a^{-1}}^k x)^{-1}
 \in Z_{\G_{N}}(q)I_{\G_{N}}^{+}\sqcup Z_{\G_{N}}(q)I_{\G_{N}}^{+}\varphi_{a^{-1}} 
\]
for some $k \in \Z$.

We first consider the case where $\varphi_{a^{-1}}^k xg\theta(\varphi_{a^{-1}}^k x)^{-1} \in Z_{\G_{N}}(q)I_{\G_{N}}^{+}$.
We take $z\in Z_{\G_{N}}(q)$ and $g_{0}\in I_{\G_{N}}^{+}$ such that $\varphi_{a^{-1}}^k xg\theta(\varphi_{a^{-1}}^k x)^{-1}=zg_{0}$.
Then we have 
\begin{align*}
\mathcal{N}\bigl(\varphi_{a^{-1}}^k xg\theta(\varphi_{a^{-1}}^k x)^{-1}\bigr)
&=\varphi_{a^{-1}}^{k}x\mathcal{N}(g)(\varphi_{a^{-1}}^{k} x)^{-1}\\
&=z/c(z)\cdot \mathcal{N}(g_{0}).
\end{align*}
On the other hand, since $\mathcal{N}(g)=h$ belongs to  $I_{\U_{N}}^{+}\subset I_{\G_{N}}^{+}$, we have
\[
\varphi_{a^{-1}}^{k}x\mathcal{N}(g)(\varphi_{a^{-1}}^{k} x)^{-1} \in  I_{\G_{N}}^{+}.
\]
Hence we know that $z/c(z)\in 1+\mfp_{E}$ by comparing the diagonal parts.
Therefore we have 
\[
z\in k^{\times}\subset \tilde{k}^{\times}\cong Z_{\G_{N}}(q).
\]
In particular, since $\omega$ is conjugate self-dual, we have $\omega^{\G_{N}}_{\omega,a,\zeta_{\omega}}(z)=1$, where $\omega^{\G_{N}}_{\omega,a,\zeta_{\omega}}$ is the central character of $\pi^{\G_{N}}_{\omega,a,\zeta_{\omega}}$.
Thus we get 
\[
\Theta^{\G_{N}}_{\omega,a,\zeta_{\omega},\theta}(g)
=\Theta^{\G_{N}}_{\omega,a,\zeta_{\omega},\theta}\bigl(\varphi_{a^{-1}}^k xg\theta(\varphi_{a^{-1}}^k x)^{-1}\bigr)
=\Theta^{\G_{N}}_{\omega,a,\zeta_{\omega},\theta}(zg_{0})
=\Theta^{\G_{N}}_{\omega,a,\zeta_{\omega},\theta}(g_{0}).
\]
Now we compute $\Theta^{\G_{N}}_{\omega,a,\zeta_{\omega},\theta}(g_{0})$.
We put $x=\diag(t_1, \ldots, t_{N})$. 
Let $(g_1, \ldots, g_{N})$ be the simple affine components of  $g_{0}$.
Then the simple affine components of 
\[
\mathcal{N}(zg_{0})=z/c(z)\cdot \mathcal{N}(g_{0})
\]
 are given by 
\[
\bigl(g_1+c(g_{N-1}), \ldots, g_{N-1}+c(g_1), g_{N}+(-1)^{N}c(g_{N})\bigr).
\]
On the other hand, since the simple affine components of $\mathcal{N}(g)=h$ as an element of $\G_{N}(F)$ are given by
\[
\begin{cases}
\bigl(h_{1}, \ldots, h_{n-1}, h_{n}, c(h_{n-1}), \ldots, c(h_{1}), h_{0}\bigr) & \text{if } N=2n,\\
\bigl(h_{1}, \ldots, h_{n}, c(h_{n}), \ldots, c(h_{1}), h_{0}\bigr) & \text{if } N=2n+1,
\end{cases}
\]
those of $\mathcal{N}(zg_{0})=\varphi_{a^{-1}}^kx\mathcal{N}(g)(\varphi_{a^{-1}}^kx)^{-1}$ are the termwise product of $(1,\ldots,1,a^{-1})$ and some cyclic permutation of
\[
\left(\frac{t_1}{t_2}h_{1}, \ldots, \frac{t_{N-1}}{t_{N}}c(h_{1}), a\frac{t_{N}}{t_1}h_{0}\right).
\]
Therefore $g_{0}$ satisfies the assumptions of Propositions \ref{prop:charTG} and \ref{prop:charTG_{o}}, 
and we have
\begin{align*}
\Theta^{\G_{N}}_{\omega,a,\zeta_{\omega},\theta}(g_{0})
&=
\begin{cases}
-\Kl^{0;n}_{\Nr(g_1+c(g_{2n-1}))\cdots \Nr(g_{n-1}+c(g_{n+1}))\Tr(g_{n})\Tr(g_{2n})}(\psi)&\text{if } N=2n,\\
\Kl^{1;n}_{\Nr(g_1+c(g_{2n}))\cdots \Nr(g_{n}+c(g_{n+1}))\Tr(\epsilon g_{2n+1})}(\psi)&\text{if } N=2n+1,
\end{cases}\\
&=
\begin{cases}
-\Kl^{0;n}_{\Nr(h_{1})\cdots\Nr(h_{n-1})h_{n}h_{0}}(\psi) & \text{if } N=2n,\\
\Kl^{1;n}_{\Nr(h_{1})\cdots\Nr(h_{n})h_{0}\epsilon}(\psi) & \text{if } N=2n+1.
\end{cases}
\end{align*}

We next consider the case where $\varphi_{a^{-1}}^k xg\theta(\varphi_{a^{-1}}^k x)^{-1} \in Z_{\G_{N}}(q)I_{\G_{N}}^{+}\varphi_{a^{-1}}$.
We take $z\in Z_{\G_{N}}(q)$ and $g_{0}\in I_{\G_{N}}^{+}$ such that $\varphi_{a^{-1}}^k xg\theta(\varphi_{a^{-1}}^k x)^{-1}=z\varphi_{a^{-1}} g_{0}$.
Then we have 
\begin{align*}
\mathcal{N}\bigl(\varphi_{a^{-1}}^k xg\theta(\varphi_{a^{-1}}^k x)^{-1}\bigr)
&=\varphi_{a^{-1}}^{k}x\mathcal{N}(g)(\varphi_{a^{-1}}^{k} x)^{-1}\\
&=z/c(z)\cdot \mathcal{N}(\varphi_{a^{-1}} g_{0}).
\end{align*}
Hence we know that $z/c(z)\in -(1+\mfp_{E})$ by comparing the diagonal parts.
Therefore we have 
\[
z\in\epsilon k^{\times}\subset \tilde{k}^{\times}\cong Z_{\G_{N}}(q).
\]
Since $\omega$ is conjugate self-dual, we have $\omega^{\G_{N}}_{\omega,a,\zeta_{\omega}}(z)=\omega(\epsilon)$.
Thus we get
\[
\Theta^{\G_{N}}_{\omega,a,\zeta_{\omega},\theta}(g)
=\Theta^{\G_{N}}_{\omega,a,\zeta_{\omega},\theta}(z\varphi_{a^{-1}}g_{0})
=\omega(\epsilon)\cdot\Theta^{\G_{N}}_{\omega,a,\zeta_{\omega},\theta}(\varphi_{a^{-1}}g_{0}).
\]
Now we compute $\Theta^{\G_{N}}_{\omega,a,\zeta_{\omega},\theta}(\varphi_{a^{-1}}g_{0})$.
We again put $x=\diag(t_1, \ldots, t_{N})$, and let $(g_1, \ldots, g_{N})$ be the simple affine components of $g_{0}$.
Then the simple affine components of $\mathcal{N}(z\varphi_{a^{-1}} g_{0})=z/c(z)\cdot \mathcal{N}(\varphi_{a^{-1}} g_{0})$ are given by
\[
\bigl(g_2+c(g_{N-1}), \ldots, g_{N-1}+c(g_2), ag_{N}+c(g_{1}), a^{-1}g_{1}+(-1)^{N}c(g_{N})\bigr).
\]
On the other hand the simple affine components of $\mathcal{N}(z\varphi_{a^{-1}}g_{0})=\varphi_{a^{-1}}^kx\mathcal{N}(g)(\varphi_{a^{-1}}^kx)^{-1}$ are described in the same way as the previous case.
Therefore $g_{0}$ satisfies the assumptions of Propositions \ref{prop:charTG'} and \ref{prop:charTG'_{o}} in the cases where $u=1$, 
and we have
\begin{align*}
&\Theta^{\G_{N}}_{\omega,a,\zeta_{\omega},\theta}(\varphi_{a^{-1}} g_{0})\\
&=
\zeta_{\omega}\cdot
\begin{cases}
\Kl^{0;n}_{\Nr(g_1+c(g_{2n}))\cdots \Nr(g_{n}+c(g_{n+1}))}(\psi) & \text{if } N=2n,\\
\Kl^{1;n}_{\Nr(g_1-\epsilon c(g_{2n+1}))\Nr(g_{2}+c(g_{2n}))\cdots \Nr(g_{n}+c(g_{n+2}))\Tr(g_{n+1})}(\psi) & \text{if } N=2n+1.
\end{cases}\\
&=
\zeta_{\omega}\cdot
\begin{cases}
\Kl^{0;n}_{\Nr(h_{1})\cdots\Nr(h_{n-1})h_{n}h_{0}}(\psi) & \text{if } N=2n,\\
\Kl^{1;n}_{\Nr(h_{1})\cdots\Nr(h_{n})h_{0}\epsilon}(\psi) & \text{if } N=2n+1.
\end{cases}
\end{align*}
Finally, by combining Corollary \ref{cor:parity} with the right-hand side of this equality, we get the assertion.
\end{proof}

\begin{cor}\label{cor:depthbound}
The representation $\pi_{\U_{N}}$ has a nonzero $I_{\U_{N}}^{++}$-fixed vector.
\end{cor}
\begin{proof}
The assertion follows from the same argument in the proof of Corollary 5.9 in \cite{Oi:2016}.
More precisely, in the proof of Corollary 5.9 in \cite{Oi:2016}, by replacing
\begin{itemize}
\item
\cite[Proposition 4.7]{Oi:2016} with Proposition 4.2,
\item
\cite[Lemma 5.8]{Oi:2016} with Lemma 5.8, and
\item
\cite[Corollary 5.6]{Oi:2016} with Proposition 5.5,
\end{itemize}
the same argument works for our situation.
\end{proof}

\subsection{Simple supercuspidality of $\pi_{\U_{N}}$}
Let $\pi_{\omega,a,\zeta_{\omega}}^{\GL_{N}}$ and $\pi_{\U_{N}}$ be as in the previous subsection.
In this subsection, we prove that $\pi_{\U_{N}}$ is simple supercuspidal.

\begin{lem}\label{lem:hyperspecial}
Let $\bold{x}$ be a hyperspecial point of the apartment $\mathcal{A}(\U_{N}, \bfS_{\U_{N}})\cong X_{\ast}(\bfS_{\U_{N}})\otimes_{\Z}\R$ given by
\[
\begin{cases}
0 \text{ or } (\check{e}_{1}+\cdots+\check{e}_{n})/2 & \text{if } N=2n,\\
0 & \text{if } N=2n+1,
\end{cases}
\]
and $\U_{N,\bold{x}}$ the hyperspecial subgroup of $\U_{N}(F)$ associated to $\bold{x}$.
Let $y\in \U_{N}(F)$.
If $y$ satisfies $yhy^{-1} \in \U_{N,\bold{x}}$ for an affine generic element $h \in \U_{N}(F)$, then $y\in \U_{N,\bold{x}}$.
\end{lem}

\begin{proof}
Before we begin to prove the assertion, we recall that the hyperspecial subgroup $\U_{N,\bold{x}}$ attached to the hyperspecial point $\bold{x}$ as above can be described as follows:
\begin{description}
\item[the case where $\bold{x}=0$]
\[
\U_{N,\bold{x}}
=
\{h=(h_{ij})_{ij} \in\U_{N}(F) \mid \text{$h_{ij}\in\mathcal{O}_{E}$ for every $i,j$}\},
\]
\item[the case where $N=2n$ and $\bold{x}=(\check{e}_{1}+\cdots+\check{e}_{n})/2$]
\[
\U_{N,\bold{x}}
=
\biggl\{h=\begin{pmatrix}A&\varpi^{-1}B\\\varpi C&D\end{pmatrix} \in\U_{N}(F) \,\bigg\vert\, A,B,C,D\in M_{n,n}(\mathcal{O}_{E})\biggr\},
\]
where $M_{n,n}(\mathcal{O}_{E})$ is the set of $n$-by-$n$ matrices whose entries belong to $\mathcal{O}_{E}$.
\end{description}

Let $y \in \U_{N}(F)$ satisfying $yhy^{-1} \in \U_{N,\bold{x}}$ for an affine generic element $h \in I_{\U_{N}}^{+}$.
As affine genericity is preserved by $I_{\U_{N}}$-conjugation, any element of $\U_{N,\bold{x}}yI_{\U_{N}}$ satisfies the same condition as $y$.
It suffices to show $\U_{N,\bold{x}}yI_{\U_{N}} \subset \U_{N,\bold{x}}$.

From the decomposition in \cite[Theorem 1.4]{MR3481263}, we have
\[
\U_{N,\bold{x}}\backslash \U_{N}(F)/I_{\U_{N}} \cong \bigl(\bigl(N_{\bfS_{\U_{N}}}\cap \U_{N,\bold{x}}\bigr)/T_{\U_{N}}^{0}\bigr)\big\backslash\bigl(N_{\bfS_{\U_{N}}}/T_{\U_{N}}^{0}\bigr),
\]
where $N_{\bfS_{\U_{N}}}$ is the $F$-valued points of the normalizer $\mathbf{N}_{\bfS_{\U_{N}}}$ of $\bfS_{\U_{N}}$ in $\U_{N}$.
The right-hand side of this equality is represented by
\[
\begin{cases}
\big\{\diag(\varpi^{r_1}, \ldots, \varpi^{r_n}, \varpi^{-r_n}, \ldots, \varpi^{-r_1}) \mid r_1, \ldots, r_n\in\Z \big\} & \text{if } N=2n,\\
\big\{\diag(\varpi^{r_1}, \ldots, \varpi^{r_n},1, \varpi^{-r_n}, \ldots, \varpi^{-r_1}) \mid r_1, \ldots, r_n\in\Z \big\}& \text{if } N=2n+1.
\end{cases}
\]
Hence we may assume 
\[
y
=\diag(t_1, \ldots, t_{N})
=\begin{cases}
\diag(\varpi^{r_1}, \ldots, \varpi^{r_n}, \varpi^{-r_n}, \ldots, \varpi^{-r_1}) & \text{if } N=2n,\\
\diag(\varpi^{r_1}, \ldots, \varpi^{r_n},1, \varpi^{-r_n}, \ldots, \varpi^{-r_1}) & \text{if } N=2n+1.
\end{cases}
\]

Since $(yhy^{-1})_{ij}=t_i/t_j\cdot h_{ij}$, and $\bold{x}$ equals either $0$ or $(\check{e}_{1}+\cdots+\check{e}_{n})/2$, we have the following inequalities:
\[
r_1-r_2 \geq0,\quad\ldots,\quad
r_{n-1}-r_n \geq0,
\]
\[
\begin{cases}
2r_n \geq-1 & \text{if }N=2n\\
r_n \geq 0 & \text{if }N=2n+1
\end{cases},\quad\text{and}\quad
-2r_1 \geq-1.
\]
Therefore $(r_1, \ldots, r_n)$ is equal to $(0, \ldots, 0)$.

Hence we have $\U_{N,\bold{x}}yI_{\U_{N}}=\U_{N,\bold{x}}I_{\U_{N}}=\U_{N,\bold{x}}$, and this completes the proof.
\end{proof}

\begin{prop}\label{prop:simplity}
The representation $\pi_{\U_{N}}$ is simple supercuspidal.
\end{prop}

\begin{proof}
We first note that $I_{\U_{N}}^{++}$ is the $(\frac{1}{N}+)$-th Moy--Prasad filtration of $I_{\U_{N}}$.
Therefore, from Corollary \ref{cor:depthbound}, the depth of $\pi_{\U_{N}}$ is not greater than $\frac{1}{N}$.
Since $\frac{1}{N}$ is the minimal positive depth of representations of $\U_{N}(F)$, we can conclude that $\pi_{\U_{N}}$ is depth-zero or simple supercuspidal (see Corollary B.8 and Proposition B.10 in \cite{Oi:2018}).
Thus our task is to eliminate the possibility that the depth of $\pi_{\U_{N}}$ is zero.

Let us suppose that $\pi_{\U_{N}}$ is depth-zero supercuspidal.
Since every tempered $L$-packet contains a generic representation with respect to a fixed Whittaker datum (Corollary 9.2.4 in \cite{MR3338302}), $\pi_{\U_{N}}$ is generic.
Then, by Lemma 6.1.2 in \cite{MR2480618}, the generic depth-zero supercuspidal representation $\pi_{\U_{N}}$ can be obtained by the compact induction of a representation of the reduction of the hyperspecial subgroup $\U_{N,\bold{x}}$ of $\U_{N}(F)$ for some hyperspecial vertex $\bold{x}\in\mathcal{A}(\U_{N}, \bfS_{\U_{N}})$.
Since every hyperspecial vertex of the building $\mathcal{B}(\U_{N})$ of $\U_{N}$ is conjugate to
\[
\mathcal{A}(\U_{N}, \bfS_{\U_{N}})\ni
\begin{cases}
0 \text{ or } (\check{e}_{1}+\cdots+\check{e}_{n})/2 & \text{if } N=2n,\\
0 & \text{if } N=2n+1,
\end{cases}
\]
we may assume $\bold{x}$ is one of them (see \cite[Section 4]{MR546588}).

Let $\U_{N,\bold{x}}^{+}$ be the pro-unipotent radical of $\U_{N,\bold{x}}$.
Let $\pi_{\U_{N}} \cong \cInd_{\U_{N,\bold{x}}}^{\U_{N}(F)} \dot{\rho}$, where $\dot{\rho}$ is the inflation of a representation $\rho$ of $\U_{N,\bold{x}}/\U_{N,\bold{x}}^{+}$ to $\U_{N,\bold{x}}$.
Then, by the character formula (Theorem \ref{thm:CF}), we have
\[
\Theta_{\pi_{\U_{N}}}\bigl(\mathcal{N}(1+\varphi_{\tilde{u}})\bigr) = 
\sum_{x\in \U_{N,\bold{x}} \backslash \U_{N}(F) / \U_{N,\bold{x}}}
\sum_{\begin{subarray}{c} y\in \U_{N,\bold{x}} \backslash \U_{N,\bold{x}}x\U_{N,\bold{x}}\\ y\mathcal{N}(1+\varphi_{\tilde{u}})y^{-1} \in \U_{N,\bold{x}} \end{subarray}} \tr \dot{\rho}\bigl(y\mathcal{N}(1+\varphi_{\tilde{u}})y^{-1}\bigr)
\]
for any $u\in k^{\times}$.
Here, for $u\in k^{\times}$, we put 
\[
\tilde{u}=
\begin{cases}
u& \text{if } N=2n,\\
u\epsilon^{-1}& \text{if } N=2n+1.
\end{cases}
\]
The index set of this sum consists of only one element by Lemma \ref{lem:hyperspecial}, and we have
\[
\Theta_{\pi_{\U_{N}}}\bigl(\mathcal{N}(1+\varphi_{\tilde{u}})\bigr)=\tr\dot{\rho}\bigl(\mathcal{N}(1+\varphi_{\tilde{u}})\bigr).
\]

We first consider the case where $\bold{x}=0$.
Then the image of $\mathcal{N}(1+\varphi_{\tilde{u}})$ in $\U_{N,\bold{x}}/\U_{N,\bold{x}}^{+}$ is independent of $\tilde{u}$.
Hence the character $\Theta_{\pi_{\U_{N}}}(\mathcal{N}(1+\varphi_{\tilde{u}}))$ is constant on $u \in k^{\times}$.
However we have 
\[
\Theta_{\pi_{\U_{N}}}\bigl(\mathcal{N}(1+\varphi_{\tilde{u}})\bigr)
=
\begin{cases}
-\Kl^{0;n}_{2^{2n}u}(\psi) & \text{if } N=2n,\\
\Kl^{1;n}_{2^{2n+1}u}(\psi) & \text{if } N=2n+1,
\end{cases}
\]
by Proposition \ref{prop:CR}.
In particular, this is not constant on $u$ (see, e.g., \cite[Proposition A.2]{Oi:2018}).
This is a contradiction.

We next consider the case where $\bold{x}=(\check{e}_{1}+\cdots+\check{e}_{n})/2$ (note that $N=2n$ in this case).
We take $v\in \tilde{k}^{\times}$ such that $\Nr(v)=\tilde{u}$, and set 
\[
t=\diag\bigl(v, \ldots, v, c(v)^{-1}, \ldots, c(v)^{-1}\bigr) \in \U_{N,\bold{x}}.
\]
Then we have 
\[
t\mathcal{N}(1+\varphi_{\tilde{u}})t^{-1} \equiv \mathcal{N}(1+\varphi_1) \pmod{\U_{N,\bold{x}}^{+}}.
\]
Hence, by applying the character formula (Theorem \ref{thm:CF}) and Lemma \ref{lem:hyperspecial} to $t\mathcal{N}(1+\varphi_{\tilde{u}})t^{-1}$, we get 
\[
\Theta_{\pi_{\U_{N}}}\bigl(\mathcal{N}(1+\varphi_{\tilde{u}})\bigr)
=\Theta_{\pi_{\U_{N}}}\bigl(t\mathcal{N}(1+\varphi_{\tilde{u}})t^{-1}\bigr)
=\tr\dot{\rho}\bigl(\mathcal{N}(1+\varphi_1)\bigr).
\]
Thus the character $\Theta_{\pi_{\U_{N}}}(\mathcal{N}(1+\varphi_{\tilde{u}}))$ is constant on $u\in k^{\times}$.
This is a contradiction.
\end{proof}

\subsection{Main theorem}
\begin{prop}\label{prop:main}
The representation $\pi_{\U_{N}}$ of $\U_{N}(F)$ associated to $\pi_{\omega,a,\zeta_{\omega}}^{\G_{N}}$ is given by $\pi^{\U_{N}}_{\omega^{1},a}$.
Here we regard $\omega$ also as a character on $\tilde{k}^{1}$ via the isomorphism $z\mapsto z/c(z)$ from $\tilde{k}^{\times}/k^{\times}$ to $\tilde{k}^{1}$, and denote it by $\omega^{1}$.
\end{prop}

\begin{proof}
By replacing the fixed uniformizer $\varpi$, we may assume that 
\[
a=
\begin{cases}
1& \text{if }N=2n,\\
\epsilon& \text{if }N=2n+1.
\end{cases}
\]
By Proposition \ref{prop:simplity}, $\pi_{\U_{N}}$ is simple supercuspidal.
We take $\omega'\in(\tilde{k}^{1})^{\vee}$ and $b\in k^{\times}$ such that $\pi_{\U_{N}} \cong \pi^{\U_{N}}_{\omega',ab}$.
Our task is to determine $b$ and $\omega'$.

For $u \in k^{\times}$, we put
\[
\tilde{u}=
\begin{cases}
u& \text{if }N=2n,\\
u\epsilon^{-1}& \text{if }N=2n+1.
\end{cases}
\]
Then we have
\begin{align*}
\Theta_{\pi_{\U_{N}}}\bigl(\mathcal{N}(1+\varphi_{\tilde{u}})\bigr)
=\Theta^{\G_{N}}_{\omega,a,\zeta_{\omega},\theta}(1+\varphi_{\tilde{u}})
=
\begin{cases}
-\Kl_{2^{2n}u}^{0;n}(\psi) & \text{if } N=2n,\\
\Kl_{2^{2n+1}u}^{1;n}(\psi) & \text{if } N=2n+1
\end{cases}
\end{align*}
by Proposition \ref{prop:CR}.
On the other hand, since $\mathcal{N}(1+\varphi_{\tilde{u}})\in I_{\U_{N}}^{+}$ is an affine generic element with its simple affine components $(2, \ldots, 2, 2\tilde{u})$, we have
\[
\Theta_{\pi_{\U_{N}}}\bigl(\mathcal{N}(1+\varphi_{\tilde{u}})\bigr) =
\begin{cases}
-\Kl_{2^{2n}bu}^{0;n}(\psi) & \text{if } N=2n,\\
\Kl_{2^{2n+1}bu}^{1;n}(\psi) & \text{if } N=2n+1
\end{cases}
\]
by Propositions \ref{prop:charU} and \ref{prop:charU_{o}}.
Hence we get $b=1$ (see, e.g., \cite[Proposition A.3]{Oi:2018}).

For $z\in Z_{\G_{N}}(q)$, we have 
\[
\Theta_{\pi_{\U_{N}}}\bigl(z/c(z)\cdot \mathcal{N}(1+\varphi_{\tilde{u}})\bigr)=\omega'\bigl(\ol{z/c(z)}\bigr)\cdot\Theta_{\pi_{\U_{N}}}\bigl(\mathcal{N}(1+\varphi_{\tilde{u}})\bigr).
\]
From Proposition \ref{prop:CR}, we get $\omega'\bigl(\ol{z/c(z)}\bigr)=\omega(\ol{z})$.
In other words, we have $\omega'=\omega^{1}$.
\end{proof}

From this proposition and Corollary \ref{cor:parity}, for $(\omega^{1},b)\in\SSC(\U_{N})$, we know that the lift of $\pi^{\U_{N}}_{\omega^{1},b}$ to $\G_{N}(F)$ via $\xi_{+1}$ is given by $\pi^{\G_{N}}_{\omega,b,(-1)^{N-1}\omega^{1}(-1)}$.
Here we regard $\omega^{1}\in(\tilde{k}^{1})^{\vee}$ as an element of $(\widetilde{k}^{\times}/k^{\times})^{\vee}$ and denote it by $\omega$.
Then, by Lemmas \ref{lem:csdLpar} and \ref{lem:csdssc}, we also know that the lift of $\pi^{\U_{N}}_{\omega^{1},b}$ to $\G_{N}(F)$ via $\xi_{-1}$ is given by $\pi^{\G_{N}}_{\omega,b,(-1)^{N}\omega^{1}(-1)}$.

In summary, we get the following result.

\begin{thm}[Main theorem]\label{thm:main}
Let $(\omega^{1},b)\in\SSC(\U_{N})$.
Then the following hold.
\begin{enumerate}
\item
The $L$-packet containing the simple supercuspidal representation $\pi^{\U_{N}}_{\omega^{1},b}$ of $\U_{N}(F)$ is a singleton.
In particular, the character of $\pi^{\U_{N}}_{\omega^{1},b}$ is stable.
\item
The endoscopic lift of the simple supercuspidal representation $\pi^{\U_{N}}_{\omega^{1},b}$ of $\U_{N}(F)$ to $\G_{N}(F)$ via the $L$-embedding $\xi_{\kappa}$ is again simple supercuspidal, and given by $\pi^{\G_{N}}_{\omega,b,(-1)^{N-1}\kappa\omega^{1}(-1)}$.
\end{enumerate}
\end{thm}

\begin{rem}\label{rem:L-par}
From this result, we know that the $L$-parameter of $\pi^{\U_{N}}_{\omega^{1},b}$ is equal to that of its endoscopic lift (as representations of $W_{F}$).
On the other hand, $L$-parameters of simple supercuspidal representations of general linear groups have been determined explicitly by the works \cite{MR2148193} and \cite{Imai:2015aa}.
Therefore we can get an explicit description of $L$-parameters of simple supercuspidal representations of $\U_{N}(F)$. 
\end{rem}

%
%

\end{document}